\documentclass{article}
\usepackage{amsfonts}
\usepackage{amssymb}
\usepackage{cite}
\usepackage{graphicx}

\newif\ifpics
\picstrue 

\usepackage{color} 
   \definecolor{cites}{rgb}{0.50 , 0.00 , 0.00}  
   \definecolor{urls} {rgb}{0.00 , 0.00 , 0.50}  
   \definecolor{links}{rgb}{0.00 , 0.00 , 0.50}   

\usepackage{hyperref} 
\hypersetup{  
      colorlinks=true,   
      citecolor=cites,   
      urlcolor=urls,     
      linkcolor=links,   
      pdfstartview=FitH,       
      bookmarksopen=false      
}

\newcommand\C{{\mathbb C}}
\newcommand\D{{\mathbb D}}

\newcommand\N{{\mathbb N}}

\newcommand\R{{\mathbb R}}

\newcommand\Z{{\mathbb Z}}

\newcommand\eps{{\varepsilon}}
\newcommand\spec{{\rm spec}\,}  
\newcommand\specn{{\rm spec}}   
\newcommand\speps{{\rm Spec}_\eps}
\newcommand\spess{{\rm spec}_{\rm ess}\,}
\newcommand\sppt{{\rm spec}_{\rm point}^\infty\,}
\newcommand\sppte{{\rm spec}_{{\rm point},\eps}\,}
\newcommand\sppten{{\rm spec}_{{\rm point},\eps}}
\newcommand\spptp{{\rm spec}_{\rm point}^p\,}

\newcommand\opsp{{\sigma^{\sf op}}}
\newcommand\opspp{{\sigma_+^{\sf op}}}
\newcommand\opspm{{\sigma_-^{\sf op}}}
\newcommand\opsppm{{\sigma_\pm^{\sf op}}}
\newcommand\dist{{\rm dist}}

\newcommand\ri{{\rm i}}
\newcommand\re{{\rm e}}
\newcommand\ind{{\rm ind\,}}

\newcommand{\rdots}{.\hspace{.1em}\raisebox{.8ex}{.\hspace{.1em}\raisebox{.8ex}{.}}}

\newtheorem{theorem}{Theorem}[section]
\newtheorem{lemma}[theorem]{Lemma}
\newtheorem{corollary}[theorem]{Corollary}

\newtheorem{definition}[theorem]{Definition}

\newenvironment{example}
 {\par\noindent\refstepcounter{theorem}{\bf Example \thetheorem}}
 {\raisebox{1mm}{\framebox{}}\pagebreak[2]}

\newcommand\Proofend{\rule{2mm}{2mm}}

\newenvironment{proof}
 {\par\noindent{\bf Proof.}}
 {\Proofend\pagebreak[2]}

\begin{document}
\title{\bf On the Spectra and Pseudospectra of a Class of Non-Self-Adjoint Random Matrices and Operators}
\author{{\sc Simon N. Chandler-Wilde}\footnote{Email: {\tt S.N.Chandler-Wilde@reading.ac.uk}},\quad {\sc Ratchanikorn Chonchaiya}\footnote{Email: {\tt ratchanikorn@buu.ac.th}}\\[1mm] and\quad {\sc Marko Lindner}\footnote{Email: {\tt lindner@tuhh.de}}}
\date{\today}
\maketitle
\begin{quote}
\renewcommand{\baselinestretch}{1.0}
\footnotesize {\sc Abstract.} In this  paper we develop and apply methods for the spectral analysis of non-self-adjoint tridiagonal infinite and finite random matrices, and for the spectral analysis of analogous deterministic matrices which are pseudo-ergodic in the sense of E.~B.~Davies (Commun.\ Math.\ Phys.\ 216 (2001), 687--704). As a major application to illustrate our methods we focus on the ``hopping sign model'' introduced by J.~Feinberg and A.~Zee (Phys.\ Rev.\ E
59 (1999), 6433--6443), in which the main objects of study are random tridiagonal matrices which have zeros on the main diagonal and random $\pm 1$'s as the other entries. We explore the relationship between spectral sets in the finite and infinite matrix cases, and between the semi-infinite and bi-infinite matrix cases, for example showing that the numerical range and $p$-norm $\eps$-pseudospectra ($\eps>0$, $p\in [1,\infty]$) of the random finite matrices converge almost surely to their infinite matrix counterparts, and that the finite matrix spectra are contained in the infinite matrix spectrum $\Sigma$. We also propose a sequence of inclusion sets for $\Sigma$ which we show is convergent to $\Sigma$, with the $n$th element of the sequence computable by calculating smallest singular values of (large numbers of) $n\times n$ matrices. We propose similar convergent approximations for the 2-norm $\eps$-pseudospectra of the infinite random matrices, these approximations sandwiching the infinite matrix pseudospectra from above and below.
\end{quote}


\section{Introduction}
In the last fifteen years there have been many studies of the spectra and pseudospectra of infinite random tridiagonal matrices in the non-self-adjoint case, and of the relationship of the spectral sets of these infinite matrices to those of corresponding large finite random $n\times n$ matrices (see e.g.\ \cite{HatanoNelson1997,FeinZee97,NelsonShnerb1998,FeinZee99,GoldKoru,TrefContEmb,Davies2001:SpecNSA,BoeEmLi,BoeEmSok02,BoeEmSok03a,BoeEmSok03b,HolzOrlZee,TrefEmbBook,BoeGru,Martinez2007,LindnerRoch2010}
and the references therein). In this paper we contribute to this literature, introducing new methods of analysis and computation
 with emphasis throughout, as a major case study, on applying these techniques to understand the ``hopping sign model'' introduced by Feinberg and Zee \cite{FeinZee99}, further studied in Holz, Orland and Zee \cite{HolzOrlZee}, by ourselves previously in \cite{CWChonchaiyaLindner2011}, and see also \cite{CicutaContediniMolinari2000,CicutaContediniMolinari2002} and \cite[Section 37]{TrefEmbBook}. In this model the main object of study is the order $n$ tridiagonal matrix given, for $n\ge 2$, by
$$
A^b_n = \left(\begin{array}{ccccc}
0 & 1& & &  \\
b_1 & 0 & 1 & & \\
& b_2 & 0 & \smash{\ddots} & \\
 & & \ddots & \ddots & 1 \\
 & & & b_{n-1} & 0 \end{array} \right),
$$
where $b = (b_1,\dots,b_{n-1})\in\C^{n-1}$ and each $b_j=\pm1$. (For $n=1$ we set  $A^b_n=(0)$.)

\noindent
\ifpics{  
\begin{center}
\includegraphics[width=1\textwidth]{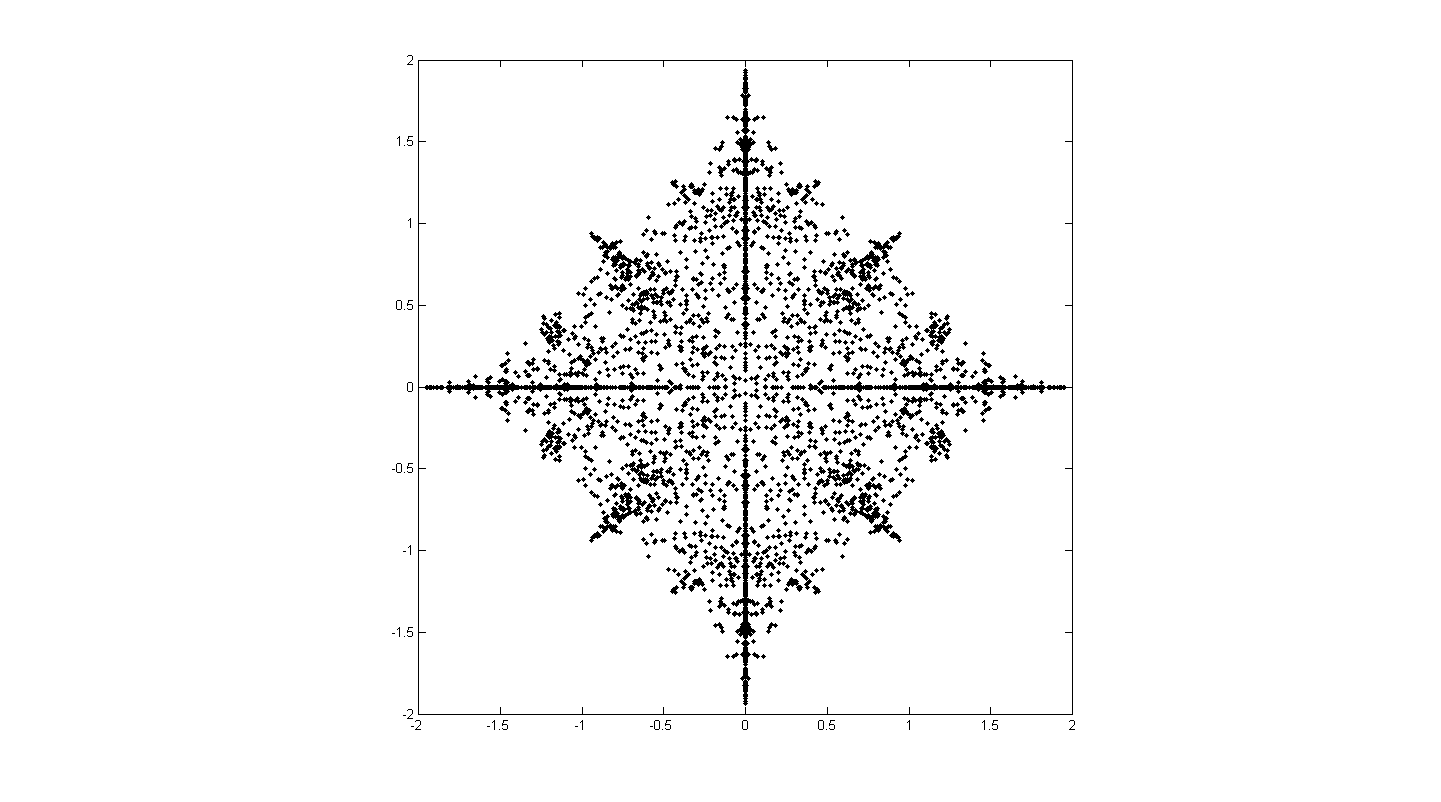}
\end{center}
}\fi 
\begin{figure}[h]
\caption{\footnotesize A plot of $\spec A^b_n$, the set of eigenvalues of $A^b_n$, for a randomly chosen $b\in \{\pm 1\}^{n-1}$, with $n=5000$ and the components $b_j$ of $b$ independently and identically distributed, with each $b_j$ equal to 1 with probability $1/2$. Note the symmetry about the real and imaginary axes by Lemma \ref{lem_symm} below, and that the spectrum is contained in the square with corners at $\pm 2$ and $\pm 2\ri$ by Lemma \ref{lem_nr} below.}  \label{fig:5000random}
\end{figure}

The objectives we set ourselves in this paper are to understand the behaviour of the spectrum and pseudospectrum of the matrix $A_n^b$,  the spectrum and pseudospectrum of the corresponding semi-infinite and bi-infinite matrices, and the relationship between these spectral sets in the finite and infinite cases. Emphasis will be placed on asymptotic behaviour of the spectrum and pseudospectrum of the finite matrix $A^b_n$ as $n\to\infty$, and we will be interested particularly in the case when the $b_j$ are random variables, for example independent and identically distributed (iid), with $\mathrm{Pr}(b_j=1)=0.5$ for each $j$. (A visualisation of $\spec A^b_n$ for a realisation of this random matrix with $n=5000$ is shown in Figure \ref{fig:5000random}; cf.\ \cite{FeinZee99}.) To be more precise, we will focus on the case when the vector $b\in\{\pm1\}^{n-1}$ is the first $n-1$ terms of an infinite sequence $(b_1,b_2,\dots)$, with each $b_j=\pm 1$,  which is {\em pseudo-ergodic} in the sense introduced by Davies \cite{Davies2001:PseudoErg}, which simply means that every finite sequence of $\pm 1$'s appears somewhere in $(b_1,b_2,\dots)$ as a consecutive sequence. If the $b_j$ are random variables then, for a large class of probability distributions for the $b_j$, in particular if each $b_j$ is iid with $\mathrm{Pr}(b_j=1)\in (0,1)$ for each $j$, it is clear that the sequence $(b_1,b_2,\dots)$ is pseudo-ergodic almost surely (with probability one). Thus, although pseudo-ergodicity is a purely deterministic property, our results assuming pseudo-ergodicity have immediate and significant corollaries for the case when $A^b_n$ is a random matrix.

Our interest in studying this problem is in making a contribution to the understanding of the relationship between the spectral properties of finite random matrices and corresponding infinite random matrices in the difficult {\em non-normal} case. (We note that $A^b_n$ is self-adjoint only in the special case that each $b_j=1$, and it is an easy calculation that $A^b_n$ is normal, i.e.\ $A^b_n$ commutes with its transpose, only if $b_1=b_2=\dots=b_{n-1}$.)  For an interesting introduction to the behaviour of random matrices in the non-normal case see \cite{TrefEmbBook}. Our focus in this paper is on the particular matrix $A^b_n$ and especially its infinite counterparts, but in the course of this investigation we develop and apply methods applicable to the study of spectral sets for the much larger classes of infinite tridiagonal or banded matrices.

Our study of the particular matrix $A^b_n$, with each $b_j=\pm 1$, is motivated by interest expressed in this  class of random matrix in the physics literature \cite{FeinZee97,HolzOrlZee,CicutaContediniMolinari2000,CicutaContediniMolinari2002}. Despite this interest there are so far no rigorous mathematical results on the behaviour of the spectrum of $A^b_n$ in the limit as $n\to\infty$. This paper makes steps in this direction. A further motivation for studying the particular matrix class $A^b_n$ is that rigorous results are available on the asymptotics of the spectrum and resolvent norm for a related class of matrices, offering some hope that progress might be possible in this case also. This related class is the case when, rather than the first sub-diagonal consisting of random $\pm1$'s, the diagonal has random $\pm1$'s. Of course, the matrix is then upper-triangular, so that many computations become explicit; in particular the spectrum of the finite matrix is just $\{1,-1\}$ and the spectra of the corresponding infinite matrices can be explicitly calculated: see \cite{TrefContEmb,CWLi2008:Memoir,LiBiDiag} for details. We shall see that the situation in the case studied in this paper is, in a number of respects, rather richer and the analysis more delicate. At the same time in a number of respects our results are more complete: for example, we are able to prove convergence of the pseudospectra of $A^b_n$ to those of the corresponding infinite matrices, and to do this not just in a Hilbert space setting but in $p$-norm for $p\in [1,\infty]$.

The distinctive flavour of the results we develop in this paper, with their significant emphasis on pseudospectra and the relationship between finite random matrices and their infinite matrix counterparts, is in large part inspired by the paper by Trefethen, Contedini and Embree \cite{TrefContEmb}, by Part VIII on random matrices in \cite{TrefEmbBook}, and by results on convergence of the $p$-norm pseudospectra and numerical ranges of $n\times n$ Toeplitz matrices due to
Reichel, Trefethen \cite{ReichTref} and B\"ottcher \cite{Boe94} ($p=2$), and B\"ottcher, Grudsky, and Silbermann \cite{BoeGruSilb97} and Roch \cite{Roch98} ($1<p<\infty$), described more recently in the monograph of B\"ottcher and Grudsky \cite{BoeGru}. 

It is appropriate to draw attention also to a series of papers that is, in some sense, intermediate in its topic between this paper and studies of pure Toeplitz matrices and operators, namely the work of B\"ottcher, Embree and co-authors Sokolov and Lindner \cite{BoeEmSok02,BoeEmSok03a,BoeEmSok03b,BoeEmLi} on randomly perturbed Toeplitz and Laurent operators.
The paper \cite{BoeEmSok02} studies the relation between the spectra of Toeplitz and Laurent operators (i.e. semi-infinite and bi-infinite Toeplitz matrices) in the presence of localised random perturbations. The papers \cite{BoeEmSok03a,BoeEmSok03b} are about the approximation of the spectrum of the perturbed semi-infinite Toeplitz matrix by the spectra of corresponding finite submatrices. The same question for bi-infinite matrices and periodised submatrices (circulants) is the topic of \cite{BoeEmLi}. Interestingly, although the setting in \cite{BoeEmSok02,BoeEmSok03a,BoeEmSok03b,BoeEmLi} is different from that of the current paper, there are common phenomena such as the appearance of fractal structures in the spectra.


Let $\N$ denote the set of positive integers and $\Z$ the set of integers. Throughout, $\{\pm1\}^\Z$, $\{\pm1\}^\N$, and $\{\pm1\}^m$, for $m\in\N$, denote the sets of vectors in $\ell^\infty(\Z)$, $\ell^\infty(\N)$ and $\C^m$, respectively, whose entries $b_j=\pm1$. The related infinite-dimensional operators we study include the operators $A^b_+$, for $b=(b_1,b_2,\dots)\in \ell^\infty(\N)$, especially when each $b_j=\pm1$. Here $A^b_+$ acts on the sequence space $\ell^p(\N)$, for $p\in[1,\infty]$, by the action
\begin{equation} \label{Aaction}
(A^b_+ x)_i = \sum_{j\in\N} (A^b_+)_{ij} x_j, \quad i\in\N,
\end{equation}
where $(A^b_+)_{ij}= b_{i-1}\delta_{i-1,j}+\delta_{i+1,j}$ and $\delta_{ij}$ is the usual Kronecker delta. In other words, $A^b_+$ acts by multiplication by the infinite matrix
$$
A^b_+ = \left(\begin{array}{cccc}
0 & 1& &  \\
b_1 & 0 & 1 & \\
& b_2 & 0 & \smash{\ddots}  \\
& & \ddots & \ddots  \end{array} \right),
$$
which has entry $(A^b_+)_{ij}$ in row $i$, column $j$, for $i,j\in\N$. (For simplicity, we make no distinction in our notation between $A^b_+$ and its matrix representation.) A main aim of the paper will be to compute the spectrum, pseudospectrum, and numerical range of $A^b_+$ in the case when $b\in\{\pm1\}^\N$ is pseudo-ergodic. We shall also study the same properties of the corresponding operator $A^b$  which acts on $\ell^p(\Z)$, again focusing on the case when $b\in\{\pm 1\}^\Z$ is pseudo-ergodic. The action of $A^b$ is given by the same formula (\ref{Aaction}) but now with $b\in\ell^\infty(\Z)$ and with $\N$ replaced by $\Z$. In other words, $A^b$ acts by mutiplication by the bi-infinite matrix
\begin{equation} \label{eq:Ab}
A^b\ :=\ \left(\begin{array}{ccccc} \\
\ddots&\ddots\\
\ddots&0&1\\
\cline{3-3}
&b_{-1}&\multicolumn{1}{|c|}{0}&1\\
\cline{3-3}
&&b_{0}&0&\smash{\ddots}\\
&&&\ddots&\ddots
\end{array}\right),
\end{equation}
where the box marks the matrix entry at $(0,0)$.
Our results will also apply, through the application of similarity transforms,  to the more general matrices
\begin{equation} \label{Abcn}
A^{b,c}_n = \left(\begin{array}{ccccc}
0 & c_1& & &  \\
b_1 & 0 & c_2  & &\\
& b_2 & 0 & \smash{\ddots}  &\\
 & & \ddots & \ddots & c_{n-1} \\
 & & & b_{n-1} & 0 \end{array} \right),
\end{equation}
in the case when $b_j=\pm1$ and $c_j=\pm1$, and to the corresponding infinite matrices
\begin{equation} \label{eq:Abc}
A^{b,c}_+ = \left(\begin{array}{cccc}
0 & c_1& &  \\
b_1 & 0 & c_2 & \\
& b_2 & 0 & \smash{\ddots}  \\
& & \ddots & \ddots  \end{array} \right)
\quad \mbox{and}\quad
A^{b,c}\ :=\ \left(\begin{array}{ccccc}
\ddots&\ddots\\
\ddots&0&c_{-1}\\
\cline{3-3}
&b_{-1}&\multicolumn{1}{|c|}{0}&c_0\\
\cline{3-3}
&&b_{0}&0&\smash{\ddots}\\
&&&\ddots&\ddots
\end{array}\right).
\end{equation}

\subsection{The Main Results}

Let us summarise the main results that we obtain in this paper, first introducing a few key notations and definitions.
We mention that first versions of a number of the results in this paper are contained in the PhD thesis of the second author \cite{HengPhD}, and that a number of the results were announced (without proofs) in \cite{CWChonchaiyaLindner2011}.

Throughout, where $B$ is a bounded linear operator on $\ell^p(S)$, for some $p\in[1,\infty]$, with $S=\Z$ or $\N$, or where $B$ is a square matrix, we denote by $\spec B$ the {\em spectrum} of $B$, i.e.\ the set of $\lambda\in\C$ for which $B-\lambda I$ ($I$ the identity matrix or operator) is not invertible. (We note that the spectra of $A^b$ and $A^b_+$ do not depend on $p\in[1,\infty]$, from general results on band operators (e.g.~\cite{LiBook}); of course, when $B$ is a matrix, the spectrum is just the set of eigenvalues of $B$.) Throughout, $\|x\|_p$, for $p\in[1,\infty]$, will be our notation for the standard $p$-norm of $x$, for $x\in \ell^p(S)$, with $S=\Z$ or $\N$, or $x\in\C^m$, for some $m\in\N$. Where $B$ is an operator or matrix, $\|B\|_p$ will denote the norm of $B$ induced by the vector norm $\|\cdot\|_p$, i.e.\ $\|B\|_p := \sup_{\|x\|_p=1} \|Bx\|_p$. With this notation, following e.g.\ \cite{TrefEmbBook},
for $p\in [1,\infty]$ and $\eps>0$ we define the $\ell^p$ $\eps$-pseudospectrum of $B$, $\specn^p_\eps B$, by
$$
\specn^p_\eps B := \spec B \cup \{\lambda\in\C: \|(B-\lambda I)^{-1}\|_p > \eps^{-1}\}.
$$
When $B$ is a bounded linear operator on $\ell^p(S)$, for some $p\in[1,\infty]$ and $S=\Z$ or $\N$, in general the spectrum of $B$ is larger than the set of eigenvalues of $B$. 
We let $\spptp B$ denote the set of eigenvalues of $B$ considered as an operator on $\ell^p(S)$, i.e.\
$$
\spptp B := \{\lambda\in \C: Bx = \lambda x, \mbox{ for some }x\in\ell^p(S) \mbox{ with }x\neq 0\}.
$$

A key result we obtain on the spectra of our infinite matrices, in large part through limit operator arguments described in Section \ref{sec:lo},  is the following (cf.\ \cite{Davies2001:PseudoErg}): if
$b,c,d\in \{\pm1\}^\N$, $\tilde b, \tilde c,\tilde d\in \{\pm1\}^\Z$, and $b$, $\tilde b$, $cd$, and $\tilde c\tilde d$ are all pseudo-ergodic, then
\begin{equation} \label{eq:big}
\spec A^b_+ = \spec A_+^{c,d} = \spec A^{\tilde b} = \spec A^{\tilde c,\tilde d} = \Sigma := \bigcup_{e\in\{\pm1\}^\Z} \spec A^e = \bigcup_{e\in\{\pm1\}^\Z} \sppt A^e .
\end{equation}
One surprising aspect of this formula is that the semi-infinite and bi-infinite matrices share the same spectrum, in contrast to many of the cases discussed in \cite{TrefContEmb}, this connected to the symmetries that we explore in Section \ref{sec:nrsymm}.

We do not know a simple test for membership of the set $\Sigma$ given by this characterisation (though see Figures \ref{fig:30pics1} and \ref{fig:30pics2} below for plots of known subsets of $\Sigma$, and see Section \ref{sec:final} for an algorithm for computing approximations to $\Sigma$). But this result implies that $\spec A^b \subset \Sigma$ for every $b\in\{\pm1\}^\Z$ which gives the possibility of determining subsets of $\Sigma$ by computing $\spec A^b$ for particular choices of $b$. In particular, as recalled in Section \ref{sec:lo}, when $b$ is $n$-periodic for some $n\in\N$, i.e.\ $b_{j+n}=b_j$ for $j\in\Z$, $\spec A^b$ can be computed by calculating eigenvalues of an order $n$ matrix (a periodised version of $A^b_n$). We compute $\pi_n \subset \Sigma$, for $n=5,10,...,30$ in Section \ref{sec:lo}, where $\pi_n$ denotes the union of $\spec A_b$ over all $n$-periodic $b\in\{\pm1\}^\Z$. We speculate at the end of the paper that
\begin{equation} \label{eq:pidef}
\pi_\infty := \bigcup_{n\in\N}\pi_n
\end{equation}
is dense in $\Sigma$, and it has been shown recently in \cite{CWDavies2011} that certainly $\pi_\infty$ is dense in the unit disc $\D = \{z:|z|<1\}$, which implies that $\overline{\D}\subset \Sigma$, as established slightly earlier directly from (\ref{eq:big}) in \cite{CWChonchaiyaLindner2011}. (Throughout, $\overline{S}$ denotes the closure of a set $S\subset \C$: for an element $z\in\C$, $\bar z$ denotes the complex conjugate.)

To obtain a first upper bound on $\Sigma$ we compute the $\ell^2$-numerical range, $W(A^b)$, of $A^b$ when $b$ is pseudo-ergodic. We show that, if
$b,c,d\in \{\pm1\}^\N$, $\tilde b, \tilde c,\tilde d\in \{\pm1\}^\Z$, and $b$, $\tilde b$, $cd$, and $\tilde c\tilde d$ are all pseudo-ergodic, then
$$
W(A^b_+) = W(A_+^{c,d}) = W(A^{\tilde b}) = W(A^{\tilde c,\tilde d}) = \Delta := \{z=a+\ri b: a,b\in\R,\, |a|+|b|<2\}.
$$
Since the spectrum is necessarily contained in the closure of the numerical range, this implies that
$$
\overline{\D}\subset \Sigma\subset \overline{\Delta}.
$$
We point out that the numerical range of $A^b_n$ converges to that of $A^b$, in particular that
$W(A_n^b)\nearrow \Delta$
as $n\to\infty$, if $b$ is pseudo-ergodic. (Here and throughout, for $T_n\subset\C$ and $T\subset \C$, the notation $T_n\nearrow T$ means that $T_n\subset T$ for each $n$ and that $\dist(T,T_n)\to 0$ as $n\to\infty$, with $\dist(T,T_n)$ the Hausdorff distance defined in (\ref{eq_HD}) below.)

The largest part of the paper (Section \ref{sec:fininf}) is an investigation of the relationship between the finite and infinite matrix cases with respect to behaviour of spectra and pseudospectra. The spectral case is harder: our main result is to show that the spectra of the finite matrices are subsets of the infinite matrix spectra, precisely that, for every $n$ and every $c\in\{\pm1\}^{n-1}$,
$$
\spec A_n^c \subset \pi_{2n+2} \subset \Sigma,
$$
so that $\sigma_n := \bigcup_{c\in\{\pm1\}^{n-1}} \,\spec A_n^c \subset \pi_{2n+2}\subset \Sigma$ and
\begin{equation} \label{eqn:sig_sub}
\sigma_\infty := \bigcup_{n\in\N} \sigma_n \subset \pi_\infty \subset \Sigma.
\end{equation}
We suspect that
$
\spec A^b_n \nearrow \Sigma = \spec A^b_+
$
as $n\to\infty$, if $b\in\{\pm1\}^\N$ is pseudo-ergodic, and the numerical results in Figures \ref{fig:5000random} and \ref{fig:30pics1}, and other similar computations, are suggestive of a conjecture that $\spec A^b_n\nearrow \pi_\infty$, which set, as mentioned already, we speculate is dense in $\Sigma$.

We can prove neither of these last two conjectures about spectral asymptotics. On the other hand, our theoretical results for the pseudospectrum are fairly complete. We show first in Theorem \ref{thm_main} a pseudospectral version of (\ref{eq:big}), that, if
$b,c,d\in \{\pm1\}^\N$, $\tilde b, \tilde c,\tilde d\in \{\pm1\}^\Z$, and $b$, $\tilde b$, $cd$, and $\tilde c\tilde d$ are all pseudo-ergodic, then, for $p\in[1,\infty]$ and $\eps>0$,
$$
\specn_\eps^p A^b_+ = \specn_\eps^p A_+^{c,d} = \specn_\eps^p A^{\tilde b} = \specn_\eps^p A^{\tilde c,\tilde d} = \Sigma_\eps^p := \bigcup_{e\in\{\pm1\}^\Z} \specn_\eps^p A^e.
$$
We then show that the pseudospectra of the large finite matrices are contained in and are well-approximated by the pseudospectra of the infinite matrices, and that this works for $p$-norm pseudospectra for the full range $p\in [1,\infty]$. Precisely, for $p\in[1,\infty]$ and $\eps>0$, we show that, if $b$ is pseudo-ergodic, then
\begin{equation} \label{eq:pslim}
\specn_\eps^p A_n^b \nearrow \Sigma_\eps^p
\end{equation}
as $n\to\infty$.

This last result, linking the pseudospectra of $A^b_+$ with those of its finite sections $A^b_n$, is a somewhat unexpectedly satisfactory result. Even in the case in which the theory of the finite section method is arguably simplest and most well-understood, namely the case of the Toeplitz operator (a semi-infinite Toeplitz matrix), the limit as $n\to \infty$ of the $\eps$-pseudospectra of the $n\times n$ finite section Toeplitz matrices has been calculated only relatively recently, and only for $p\in (1,\infty)$ \cite{BoeGruSilb97,BoeGru}. Moreover, except for the special case $p=2$ (see \cite{ReichTref,Boe94}), this limit is not, in general, just the $\ell^p$ $\eps$-pseudospectrum of the Toeplitz operator, but rather the union of the $\ell^p$ and $\ell^q$ $\eps$-pseudospectra, with $p^{-1}+q^{-1}=1$. (A component of the explanation of (\ref{eq:pslim}) is that we show in Lemma \ref{lem_sim} that $\Sigma_\eps^p=\Sigma_\eps^q$ for $p^{-1}+q^{-1}=1$.)

Equation (\ref{eq:pslim}) leads to characterisations of the spectrum $\Sigma$ which, in principle, can be used for numerical approximation. Since $\bigcap_{\eps>0} \Sigma_\eps^p = \Sigma$, it holds that
\begin{equation} \label{eqn:ch_G}
\Sigma = \lim_{\eps\to 0} \Sigma_\eps^p = \lim_{\eps\to 0}\lim_{n\to\infty} \specn_\eps^p A_n^b,
\end{equation}
for every $p\in[1,\infty]$ and pseudo-ergodic $b$. However, the formula (\ref{eqn:ch_G}) is not guaranteed to give useful results for any fixed $\eps$ and $n$ as the convergence as $n\to\infty$ may be arbitrarily slow, as discussed in Section \ref{sec:final}. In that section we develop alternative, much more useful, convergent sequences of computable, upper and lower bounds for $\Sigma_\eps^2$ and a convergent sequence of computable upper bounds for $\Sigma$. We show firstly that
$$
\sigma^2_{n,\eps} := \bigcup_{c\in\{\pm1\}^{n-1}}\specn_\eps^2 A_n^c \subset \Sigma_\eps^2 \subset \sigma^2_{n,\eps+\eps_n}=\bigcup_{c\in\{\pm1\}^{n-1}}\specn_{\eps+\eps_n}^2 A_n^c,
$$
giving explicit expressions for the $\eps_n$ which satisfy that $\eps_n=O(n^{-1})$ as $n\to\infty$, and showing that $\sigma^2_{n,\eps}\nearrow \Sigma_\eps^2$ and $\sigma^2_{n,\eps+\eps_n}\searrow \Sigma_\eps^2$ as $n\to\infty$.  (The notation $T_n\searrow T$ means that $T\subset T_n$ for each $n$ and that $\dist(T,T_n)\to 0$ as $n\to\infty$.) Then, taking the intersection over all $\eps$, we deduce that
$$
\sigma_n = \bigcup_{c\in\{\pm1\}^{n-1}}\spec A_n^c \subset \Sigma \subset \overline{\sigma^2_{n,\eps_n}},
$$
and prove that
$$
\overline{\sigma^2_{n,\eps_n}}\searrow \Sigma \quad \mbox{as} \quad n\to\infty.
$$
In a substantial series of numerical calculations, we compute these convergent upper bounds $\overline{\sigma^2_{n,\eps_n}}$ for the spectrum $\Sigma$ in Section \ref{sec:final}, and through these calculations demonstrate that $\Sigma$ is a strict subset of $\overline{\Delta}$.

All these results have  implications for the behaviour of the spectral sets  of $A^b$, $A^b_+$, $A_n^b$, $A^{b,c}$, $A^{b,c}_+$, and $A^{b,c}_n$, when the entries $b_j=\pm1$ and $c_j=\pm1$ are random, and we make explicit these implications in a final Theorem \ref{thm_rand}, in the same section summarising succintly what we have established about the spectral sets $\Sigma$ and $\Sigma_\eps^p$ (Theorem \ref{thm_vfinal}), and outlining a number of open problems.

In the course of this investigation, focused on a particular operator and matrix class, we develop results for the larger classes of tridiagonal or banded finite and infinite matrices. In particular, Theorem \ref{thm_pseudo_converg} shows that, for $p\in[1,\infty]$, $\eps>0$, the $\ell^p$ $\eps$-pseudospectrum of a general, semi-infinite tridiagonal matrix is contained, for $\eps^\prime>\eps$, in the $\ell^p$ $\eps^\prime$-pseudospectrum of its $n\times n$ finite section if $n$ is sufficiently large. It also shows corresponding results relating the pseudospectra of a general bi-infinite matrix to that of its finite sections. In Section \ref{sec:lo} we employ recent work \cite{CWLi2008:FC,CWLi2008:Memoir} on limit operator methods for the study of spectral sets for very general classes of infinite matrices. We make explicit in Theorems \ref{thm_lo} and \ref{thm_losemi} the implications of this work for the essential spectrum, spectrum, and pseudospectra of bi-infinite and semi-infinite banded matrices with numerical (as opposed to operator-valued) entries. In Section \ref{sec:final} we make the first substantive application of a new method which generates sequences of inclusion sets for the spectra and pseudospectra of a tridiagonal operator, demonstrating, through this application, that these sequences of inclusion sets can in fact converge to the spectral sets that they enclose.

\subsection{Pseudospectra and the Numerical Range} \label{sec:pseudo}

We shall need throughout the paper a number of properties of the $\eps$-pseudospectra of a bounded linear operator $B$ on a Banach space $X$, and of the pseudospectra of its adjoint operator $B^*$ on the dual space $X^*$ (dual in the sense e.g. of \cite{Kato_book}, so that $X^*$ is the set of bounded anti-linear functionals, and the spectrum of $B^*$ is the complex conjugate of the spectrum of $B$). We summarise these properties in this section, pointing out how the theory of pseudospectra in the Banach space setting has recently been significantly clarified by work of Shargorodsky \cite{Shargorodsky2008}. The properties we shall need include the equivalent definitions encapsulated in the following theorem:

\begin{theorem} \label{thm_pse} The $\eps$-pseudospectrum of a bounded linear operator $B$ on a Banach space $X$ is defined, for $\eps>0$, by any one of the following equivalent definitions:
\begin{description}
\item[(i) ] $\specn_\eps B = \spec B \cup \{\lambda\in \C: \|(B-\lambda I)^{-1}\| > \eps^{-1}\}$;
\item[(ii)] $\specn_\eps B = \spec B \cup \{\lambda \in \C: \nu(B-\lambda I) < \eps\}$, where $\nu(C)$ is the lower norm of a bounded linear operator $C$, defined by $\nu(C) := \inf_{\|x\|=1} \|Cx\|$;
\item[(iii)] $\specn_\eps B$ is the union of $\spec B$ and the set $\sppte B$ of $\eps$-pseudoeigenvalues of $B$, where $\lambda$ is an $\eps$-pseudoeigenvalue if there exists $x\in X$ with $\|x\|=1$ and $\|(B-\lambda I)x\|<\eps$;
\item[(iv)] $\specn_\eps B$ is the union of $\sppte B$ and the complex conjugate of $\sppte B^*$;
\item[(v)] $\specn_\eps B = \bigcup_{\|E\|<\eps} \spec (B+E)$, the union taken over all bounded linear operators $E$ with $\|E\|<\eps$.
\end{description}
\end{theorem}

\noindent For a proof of the equivalence of (i)-(v), and a useful short introduction to the pseudospectra of linear operators on Banach spaces, see \cite[Section 4]{TrefEmbBook}. We will use the equivalence of (i)-(iv) throughout. The equivalence of the other definitions with (v), and the connection this makes with spectra of perturbed operators, is a significant motivation for the practical interest in pseudospectra. It is clear from the above definition that $\specn_\eps B$ is an open set for $\eps>0$. An elementary but important property of the lower norm is that
\begin{equation} \label{eq:lowernorm}
|\nu(A)-\nu(B)| \leq \|A-B\|,
\end{equation}
for any bounded linear operators $A$ and $B$ on $X$.

In the case when, for some $N\in\N$, $X=\C^N$ and $B$ is an $N\times N$ matrix, (i)-(v) are equivalent additionally to $\specn_\eps B= \{\lambda\in \C:\nu(B-\lambda I)<\eps\} = \sppte B$. If $\|\cdot\|=\|\cdot\|_2$, then, for every $N\times N$ matrix $A$, $\nu(A)=s_{\mathrm{min}}(A)$, the smallest singular value of $A$. Thus these definitions are additionally equivalent to \cite{TrefEmbBook}
\begin{equation} \label{eq:sing}
\specn_\eps B = \{\lambda\in\C: s_{\mathrm{min}}(B-\lambda I) < \eps \}.
\end{equation}
Note that (\ref{eq:lowernorm}) implies that
\begin{equation} \label{eq:contdep}
|s_{\mathrm{min}}(B-\lambda I)-s_{\mathrm{min}}(B- \mu I)|\leq |\lambda-\mu|, \quad \lambda,\mu\in\C.
\end{equation}
It is  equation (\ref{eq:sing}) that we use for the numerical computations of pseudospectra in Section \ref{sec:final}.

An alternative definition of the pseudospectrum is to replace the strict inequality $>$ in (i) by $\geq$, so that the $\eps$-pseudospectrum is defined to be
$$
\speps B = \spec B \cup \{\lambda\in \C: \|(B-\lambda I)^{-1}\| \geq \eps^{-1}\}.
$$
This has the attraction that $\speps B$, like $\spec B$, is a compact set for $\eps>0$. An interesting question is whether $\overline{\specn_\eps B} = \speps B$, which hinges on the question of whether or not it is possible for the norm of the resolvent of $B$, $\|(B-\lambda I)^{-1}\|$, to take a finite constant value on a open set $G\subset\C$. Let us say that the Banach space $X$ has the {\em strong maximum property} if, for every open set $G\subset \C$, every bounded linear operator $B$ on $X$, and every $M>0$, it holds that
$$
(\|(B-\lambda I)^{-1}\|\leq M, \;\forall\,\lambda\in G) \Rightarrow (\|(B-\lambda I)^{-1}\|< M, \;\forall\,\lambda\in G).
$$
If $X$ has the strong maximum property, then no bounded linear operator on $X$ can have a resolvent norm with a constant finite value on an open subset of $\C$, and it is easy to see that $\overline{\specn_\eps B} = \speps B$.
Recently, Shargorodsky \cite{Shargorodsky2008} has shown, by constructing explicit counterexamples, that not every Banach space has the strong maximum property. But
the following theorem from \cite{Shargorodsky2008}, which extends earlier work of \cite{Globevnik1976}, makes clear that the Banach spaces of relevance to this paper do have this property.

\begin{theorem} \label{thm_resl_norm} Suppose that $X$ is a Banach space which is either finite-dimensional or is such that either $X$ or $X^*$ is complex uniformly convex (as defined e.g.\ in \cite{Shargorodsky2008}). Then $X$ has the strong maximum property. In particular, $X$ has the strong maximum property if  $X$ is a Hilbert space, or if $X = \ell^p(S)$, for $S=\N$ or $\Z$ and $p\in [1,\infty]$.
\end{theorem}


  It is clear from (v) and standard operator perturbation arguments (see \cite{TrefEmbBook} for details) that,  for $0<\eps < \eps^\prime$,  $\spec B \subset \specn_\eps B \subset \specn_{\eps^\prime} B$, and that
  \begin{equation} \label{eq:thing}
  \eps\D + \spec B\subset \specn_\eps B.
  \end{equation}
  In fact \cite{TrefEmbBook} $\eps\D + \spec B = \specn_\eps B$ if $X$ is a Hilbert space and $B$ is normal, i.e.\ $BB^*=B^*B$. Further \cite{TrefEmbBook}
\begin{equation} \label{eq_int}
\spec B = \bigcap_{\eps>0} \specn_\eps B.
\end{equation}
Generalising (\ref{eq:thing}), it holds that \cite{TrefEmbBook}
\begin{equation} \label{eq_int4}
\delta\D + \specn_\eps B \subset \specn_{\delta +\eps} B, \quad \mbox{for} \quad \eps,\delta>0.
\end{equation}

 For $S,T\subset \C$, let
 \begin{equation} \label{eq_HD}
 \dist(S,T) := \max(\sup\{\dist(z,S):z\in T\}, \sup\{\dist(z,T):z\in S\}).
 \end{equation}
 (This notion of distance, when applied to compact subsets of $\C$,  is an instance of the Hausdorff distance between compact subsets of a metric space.) Given a sequence $T_n\subset \C$ and $T\subset \C$, let us write $T_n\to T$ if $\dist(T_n,T)\to 0$ as $n\to\infty$. Additionally, let us write $T_n\nearrow T$ if $T_n\to T$ and $T_n\subset T$ for each $n$, and write $T_n\searrow T$ if $T_n\to T$ and $T\subset T_n$ for each $n$.
It is an easy calculation to show that
\begin{equation} \label{eq_pseconv2}
\specn_\eps B\searrow \spec B \;\mbox{ as }\;\eps \to 0^+.
\end{equation}
 Similarly, it holds for $\eps>0$ that $\specn_{\eps^\prime} B\searrow \speps B$, as $\eps^\prime \to \eps^+$, and $\specn_{\eps^\prime}B\nearrow \specn_\eps B$, as $\eps^\prime \to \eps^-$. Thus,  in the case where $X$ has the strong maximum property so that $\overline{\specn_\eps B} = \speps B$, it holds for $\eps>0$ that
 \begin{equation} \label{eq_pseconv}
 \specn_{\eps^\prime} B\searrow \specn_\eps B, \mbox{ as }\eps^\prime \to \eps^+, \;\mbox{ and }\;\specn_{\eps^\prime}B\nearrow \specn_\eps B, \mbox{ as }\eps^\prime \to \eps^-,
\end{equation}
so that $\specn_\eps B$ depends continuously on $\eps$.

The spectrum and $\eps$-pseudospectra are connected to the numerical range. In the case that $X$ is a Hilbert space with inner product $(\cdot,\cdot)$,  and where $B$ is a bounded linear operator on $X$, the {\em numerical range } or {\em field of values} of $B$, denoted $W(B)$, is the set
$$
W(B) := \{(Bx,x):x\in X, \,\|x\|=1\}.
$$
It is well known that this numerical range is a convex set and that $\spec B\subset \overline{W(B)}$, in fact $\spec B \subset W(B)$ if $X$ is finite-dimensional. The relationship with the $\eps$-pseudospectra is that, similarly, $\specn_\eps B \subset W(B) + \eps\D$, for $\eps>0$ \cite[Section 17]{TrefEmbBook}. Let $Y$ be a closed subspace of $X$, $P:X\to Y$ orthogonal projection onto $Y$, and let $B_Y:= PB|_Y$. Then
\begin{equation} \label{eq:numr}
W(B_Y) = \{(B_Yx,x):x\in Y, \, \|x\|=1\} =\{(Bx,x):x\in Y, \, \|x\|=1\}\subset W(B).
\end{equation}
This observation is one component in the following result \cite[Theorem 3.52]{HagenRochSilbermann:C*}:

\begin{theorem} \label{thm:nrconv}
 Suppose that $X$ is a Hilbert space and that $(P_n)_{n\in\N}$ is a sequence of orthogonal projection operators on $X$ that converges strongly to the identity operator ($P_nx\to x$ as $n\to\infty$, for every $x\in X$). Then, for every bounded linear operator $B$ on $X$, where $B_n:= P_nB|_{X_n}$ with $X_n = P_n(X)$, it holds that
 $$
 W(B_n) \nearrow W(B) \quad \mbox{as} \quad n\to\infty.
 $$
\end{theorem}

\section{Results by Limit Operator Arguments} \label{sec:lo}

Let us start this section by establishing a few additional notations and definitions.
 Throughout the remainder of the paper, if $B$ is a bounded linear operator on a Banach space $X$ we will say that $B$ is Fredholm if $B(X)$, the range of $B$, is closed and if, additionally,  $\alpha(B):= \dim (\ker B)$, the dimension of the null-space of $B$, and $\beta(B) := \dim (X/B(X))$, the co-dimension of the range of $B$, are both finite, in which case we define the index of $B$ by $\ind B := \alpha(B)-\beta(B)$.
 We will let $\spess B$ denote the {\em essential spectrum} of $B$, i.e.\ the set of $\lambda\in\C$ for which $B-\lambda I$ is not Fredholm.  Let $M_b$ be the bounded linear operator which operates on the standard sequence space $\ell^p(\Z)$, for $p\in[1,\infty]$, by multiplication by $b\in\ell^\infty(\Z)$. Explicitly, for $y\in\ell^p(\Z)$,
$$
(M_b y)_j = b_jy_j, \quad j\in\Z.
$$
Moreover, for $k\in\Z$ let $V_k$ denote the shift operator defined by
$$
(V_ky)_j = y_{j-k}, \quad j\in\Z,
$$
and note that $V_jM_b= M_{V_{j}b}V_j$, for $j\in\Z$, $b\in \ell^\infty(\Z)$.
In terms of these notations, the operators $A^b$ and $A^{b,c}$, corresponding to the infinite matrices (\ref{eq:Ab}) and (\ref{eq:Abc}), can be written as
\begin{equation} \label{eq:Abrep}
A^b = V_1M_b + V_{-1} \quad \mbox{ and } \quad A^{b,c} = V_1M_b + M_cV_{-1}.
\end{equation}
We will use these notations for $b,c\in \ell^\infty(\Z)$,  but especially for $b,c\in \{\pm 1\}^\Z$.

One major tool for computing the spectrum of the infinite matrices $A^b$ and $A^{b,c}$,
with $b,c\in\ell^\infty(\Z)$, is the method of so-called limit operators
\cite{CWLi2008:Memoir,LiBook,RaRoSiBook}. In this method a bi-infinite matrix  $B$ is studied in terms of
a family of infinite matrices that represents the behaviour of $B$ at
infinity. More precisely, let $A$ be a banded matrix $A=(a_{ij})_{i,j\in\Z}$, with $\sup_{ij}|a_{ij}|<\infty$, so that the operator induced by $A$ is a bounded operator on $\ell^p(\Z)$, for all $p\in [1,\infty]$. We say that the operator induced by the matrix
$B=(b_{ij})_{i,j\in\Z}$ is a {\sl limit operator} of the operator induced by $A$
 if, for a sequence
$h_1,h_2,...$ of integers with $|h_k|\to\infty$, it holds that
\[
a_{i+h_k,j+h_k}\ \to\ b_{ij}\qquad\textrm{as}\qquad k\to\infty,
\]
for all $i,j\in\Z$. The set of all limit operators of $A$ is denoted by
$\opsp(A)$. In some instances it is useful to think of $\opsp(A)$ as the union of two  subsets, as $\opsp(A) = \opspp(A)\cup \opspm(A)$, where $\opsppm(A)$ denotes the subset of those limit operators associated with sequences $h$ with $h_k\to\pm\infty$. It is an easy consequence of the Bolzano-Weierstrass theorem and a diagonal argument that each of $\opsppm(A)$ is non-empty, and it is clear that if $B=(b_{ij})$ is a limit operator of $A$ then $\sup_{i-j=k}|b_{ij}|\leq\sup_{i-j=k}|a_{ij}|$, for every $k\in\Z$. In particular, if $A=A^{b,c}$ for some $b,c\in \{\pm1\}^\Z$ and $B$ is a limit operator of $A$, then $B= A^{\tilde b,\tilde c}$ for some $\tilde b,
\tilde c\in \{\pm 1\}^\Z$.

 The following theorem, which applies in particular to $A^b$ and to $A^{b,c}$, connects the essential spectrum with the set of limit operators. This result is a particular case of much more general results from
 \cite{CWLi2008:FC}, \cite[Theorem 6.28, Corollary 6.49]{CWLi2008:Memoir},  which extend a main theorem on limit operators going back to \cite{LaRa1985,RaRoSi1998}.
 Note that the spectrum, as an operator on $\ell^p(\Z)$, of an infinite banded matrix $A=(a_{ij})_{i,j\in\Z}$, with $\sup_{ij}|a_{ij}|<\infty$, does not depend on $p\in [1,\infty]$, and the same is true for the essential spectrum: moreover, if $\lambda \not\in \spess A$, then $\ind (A-\lambda I)$ is also independent of $p$ (see \cite{LiWiener} or \cite[Corollary 6.49]{CWLi2008:Memoir}).

\begin{theorem} \label{thm_lo} Let $A$ be a banded matrix $A=(a_{ij})_{i,j\in\Z}$, with $\sup_{ij}|a_{ij}|<\infty$. Then
\begin{equation} \label{eq:lo1}
\spess A = \bigcup_{B\in \opsp(A)} \spec B = \bigcup_{B\in \opsp(A)} \sppt B
\end{equation}
and $\specn_\eps^p B \subset \specn_\eps^p A$, for all $\eps>0$, $p\in [1,\infty]$, and $B\in \opsp(A)$.
In particular, if $A \in \opsp(A)$, in which case we say that $A$ is {\em self-similar}, then
$$
\spec A = \spess A = \bigcup_{B\in \opsp(A)} \spec B = \bigcup_{B\in \opsp(A)} \sppt B \;\mbox{ and }\; \specn_\eps^p A = \bigcup_{B\in \opsp(A)} \specn_\eps^p B,
$$
for $\eps>0$ and $p\in [1,\infty]$.
\end{theorem}
Recall that $\sppt B$ is the set of eigenvalues of $B$ in $\ell^\infty(\Z)$, so that $\lambda \in \sppt B$ iff $\lambda x = Bx$ has a non-trivial bounded solution $x$.

One case where $A^{b,c}$ is self-similar is where $(b,c)$ is periodic with some period $n\in \N$, i.e.\
\begin{equation} \label{per}
b_{j+n} = b_j, \quad c_{j+n}=c_j, \quad j\in\Z.
\end{equation}
In this case the above theorem applied to $A^{b,c}$ reduces to $\spec A^{b,c} = \spess A^{b,c} = \sppt A^{b,c}$, and in fact it is well-known further, e.g.\ \cite{Davies2007:Book}, that if $\lambda\in \spec A^{b,c}$ then $\lambda x = A^{b,c}x$ has a solution which is not only bounded but also quasi-periodic, i.e.\, for some $\alpha\in \C$ with $|\alpha|=1$, $x_{k+n} = \alpha x_k$, $k\in\Z$. It is easy to see that this implies that
\begin{equation} \label{perspec}
\spec A^{b,c} = \bigcup_{|\alpha|=1} \spec\left(A_n^{b,c}+B_{n,\alpha}^{b,c}\right),
\end{equation}
where $A_n^{b,c}$ is given by (\ref{Abcn}) (with $A_1^{b,c}:=(0)$) and $B_{n,\alpha}^{b,c}$ is the $n\times n$ matrix whose entry in row $i$, column $j$ is $\delta_{i,n}\delta_{j,1} \alpha c_n+\delta_{i,1}\delta_{j,n} \alpha^{-1} b_n$, where $\delta_{ij}$ is the Kronecker delta. We will abbreviate $B_{n,\alpha}^{b,c}$ as $B_{n,\alpha}^{b}$ in the case that $c = (1,...,1)$.

An important case where $A^{b}$ is self-similar is where $A^{b}$ is {\em pseudo-ergodic} in the sense of Davies \cite{Davies2001:PseudoErg}. The following is a specialisation of the definition from  \cite{Davies2001:PseudoErg}.

\begin{definition} \label{pse} Call $b\in \{\pm1\}^\Z$ and the
operator $A^{b}$  {\em pseudo-ergodic} if, for every $N\in \N$ and every $w \in \{\pm 1\}^N$, there exists $J\in \Z$ such that
$b_{n+J} = w_n$, for $n=1,...,N$.
\end{definition}

We see from this definition that $A^b$ is pseudo-ergodic if and only if every finite sequence of $\pm1$'s appears somewhere in the bi-infinite sequence $b$. The significance of this definition is that, for many cases where the entries $b_n$ are random variables, the sequence $b$ is pseudo-ergodic with probability one. In particular, the following lemma follows easily from the Second Borel Cantelli Lemma (e.g.\ \cite[Theorem 8.16]{CapinskiKopp}), the argument sometimes called the `Infinite Monkey Theorem'.
\begin{lemma} \label{imt}
If the matrix entries $b_n$, for $n\in\Z$, are iid random variables taking the values $\pm 1$ with $\mathrm{Pr}(b_n=1)\in (0,1)$, then $A^b$ is pseudo-ergodic with probability one.
\end{lemma}

The link to limit operators is provided by the following lemma (see \cite[Lemma 6]{Davies2001:PseudoErg}, \cite[Corollary 3.70]{LiBook} or \cite[Theorem 7.6]{CWLi2008:Memoir}):

\begin{lemma} \label{pe_lo}
For $b\in\{\pm1\}^\Z$, $A^{b}$ is {\em pseudo-ergodic} if and only if $\opsp(A^b) = \{A^c:c\in\{\pm 1\}^\Z\}$.
\end{lemma}

Combining this lemma with Theorem \ref{thm_lo} gives the following characterisation of the spectrum and pseudospectrum of $A^b$ in the case when $b$ is pseudo-ergodic:
\begin{theorem} \label{thm_lo2} If  $b\in\{\pm1\}^\Z$ and $A^b$ is pseudo-ergodic, then
\begin{equation} \label{eq:spm}
\spec A^{b} = \spess A^{b} = \bigcup_{c\in \{\pm1\}^\Z} \spec A^c = \Sigma := \bigcup_{c\in \{\pm1\}^\Z} \sppt A^c
\end{equation}
and
\begin{equation}
\specn_\eps^p\,A^{b}\ = \Sigma_\eps^p := \bigcup_{c\in\{\pm1\}^\Z} \specn_\eps^p A^c, \label{eq:spm2}
\end{equation}
for $\eps>0$ and $p\in [1,\infty]$.
\end{theorem}

Limit operator ideas, the ``Infinite Monkey'' argument and the validity of
the first two ``='' signs in (\ref{eq:spm}) are not new in the spectral
theory of random matrices (see e.g.
\cite{CarmonaLacroix,Davies2001:SpecNSA,Davies2001:PseudoErg,GoldKoru,PasturFigotin}). Equation (\ref{eq:spm2}) is previously shown, for a general class of pseudo-ergodic operators for the case $p=2$ in \cite{Davies2001:PseudoErg}. What is more recent
is the third ``='' sign in the first of equations (\ref{eq:spm}) and the extensions to $p\in [1,\infty]$, these
 shown in \cite{CWLi2008:FC} and  \cite[Theorem 6.28, 7.6]{CWLi2008:Memoir}.

Note that the above theorem shows that the spectrum of $A^b$ is the same set $\Sigma$ for every pseudo-ergodic $b\in\{\pm1\}^\Z$, and that $\spec A^c \subset \Sigma$ for every $c\in \{\pm 1\}^\Z$, and that similar statements hold for the pseudospectrum $\specn_\eps^p A^b$. In particular, $\spec A^c \subset \Sigma$ if $c\in \Pi_n$, for some $n\in\N$, where $\Pi_n :=\{c\in \{\pm1\}^\Z: c \mbox{ is $n$-periodic}\}$. Thus
\begin{equation} \label{eq:spperntri}
\pi_n\ :=\ \bigcup_{c\in\Pi_n}\spec A^c\ =\ \bigcup_{c\in\Pi_n}\sppt
A^c \subset \Sigma,
\end{equation}
for every $n\in\N$: this is informative as $\pi_n$ can be computed explicitly by (\ref{perspec}) as the union of eigenvalues of $n\times n$ matrices. The following lemma carries out this computation for $n=1,2,3$.
\begin{lemma} \label{periods123}
If $b\in \Pi_1$ with $b_0=1$, then $\spec A^b = [-2,2]$ and $\spec A^{-b}= \ri [-2,2]$. If $b\in \Pi_2\setminus \Pi_1$ then $\spec A^b = \tau_2 := \{x\pm\ri x: -1\leq x\leq 1\}$. If $b\in \Pi_3$, $b_0=b_1=1$, and $b_2=-1$, then
$$
\spec A^b = \tau_3 := \ri[-1,1] \cup \{x+\ri y:-1/2\leq y\leq 1/2, \; x^2 = 1+3y^2\}
$$
while $\spec A^{-b} = \ri \tau_3$. Thus
$$
\pi_1 = [-2,2]\cup \ri[-2,2], \quad \pi_2 = \pi_1  \cup \tau_2, \quad \pi_3 = \pi_1 \cup \tau_3 \cup \ri\tau_3.
$$
Note that $\max_{\lambda\in \pi_1} |\lambda| = 2$ while $\max_{\lambda\in \tau_j}|\lambda|=\sqrt{2}$, for $j=2,3$. For $j=2$ this maximum is achieved at $\pm 1\pm \ri$, while for $j=3$ this maximum is achieved at $\pm \sqrt{7}/2 \pm \ri/2$.
\end{lemma}
\begin{proof} If $b\in \Pi_1$ with $b_0=\beta = \pm 1$ then, from (\ref{perspec}), $\spec A^b = \cup_{|\alpha|=1} \spec B^b_{1,\alpha} = \{\re^{\ri\theta} + \re^{-\ri \theta} \beta :\theta \in \R\}$. So  $\spec A^b= [-2,2]$ if $\beta=1$ and $\spec A^b=\ri[-2,2]$ if $\beta = -1$, and $\pi_1=[-2,2]\cup \ri[-2,2]$.

If $b\in \Pi_2\setminus \Pi_1$ then, from (\ref{perspec}), where $\beta = b_1 = \pm 1$,
$$
\spec A^b = \bigcup_{\theta\in\R} \spec \left(
                                         \begin{array}{cc}
                                           0 & 1 -\re^{-\ri\theta}\beta \\

                                           \beta +\re^{\ri\theta} & 0 \\
                                         \end{array}
                                       \right) = \{\lambda\in \C: \lambda^2 = 2\ri \sin \theta, \; \theta\in\R\}.
$$
Thus $\spec A^b = \tau_2$ and $\pi_2 = \pi_1 \cup \tau_2$.

If $b\in \Pi_3$, $b_0=b_1=1$, and $b_2=-1$, then, from (\ref{perspec}),
$$
\spec A^b = \bigcup_{\theta\in\R} \spec \left(
                                         \begin{array}{ccc}
                                           0 & 1 & -\re^{-\ri\theta} \\
                                           1 & 0 & 1 \\
                                           \re^{\ri\theta} & 1 & 0 \\
                                         \end{array}
                                       \right) = \{\lambda\in \C: \lambda^3-\lambda = -2\ri\sin\theta, \; \theta\in\R\}.
$$
Writing $\lambda = x+\ri y$, we see that $\lambda^3-\lambda = -2\ri\sin\theta$, for some $\theta\in\R$, iff
$$
x(x^2-3y^2-1)= 0 \quad \mbox{and} \quad 3x^2y-y^3-y \in [-2,2].
$$
But this implies that either $x=0$ and $y^3+y\in [-2,2]$, or $x^2 = 3y^2+1$ and $8y^3+2y\in [-2,2]$, and it follows that $\spec A^b = \tau_3$. That $\spec A^{-b}= \ri\tau$ can be shown similarly, or follows from Lemma \ref{lem_symm} below. Since $c\in \Pi_3$ iff $c= \pm V_j b$ for $j=0,1$ or 2, it follows that $\pi_3 = \pi_1\cup \tau_3\cup \ri\tau_3$.
\end{proof}

In Figure \ref{fig:30pics1} we plot $\pi_n$ for $n=5,10,...,30$, with $\pi_n$ computed numerically in Matlab using the characterisation (\ref{perspec}) (see \cite{CWChonchaiyaLindner2011} for small plots of $\pi_n$ for $n=1,2,...,30$). For each $n$ the set $\pi_n$, by the characterisation (\ref{perspec}), consists of $k\leq n2^n$ analytic arcs, and $\pi_n\subset \Sigma$. The visual impression that might be taken from this sequence of plots is that $\pi_n$ ``fills out'' a large part of the square $\Delta :=\{x+\ri y: x,y\in \R, \, |x|+|y| < 2\}$ as $n\to\infty$. But of course $\pi_\infty := \cup_{n\in\N}\pi_n$ is a countable union of analytic arcs, so that $\pi_\infty$ has (two-dimensional) Lebesgue measure 0. Thus almost every point in $\Delta$ is not in $\pi_\infty$ and so is not one of the points in the plots in Figure \ref{fig:30pics1}. Thus these figures provide no evidence that the Lebesgue measure of $\Sigma$ is any larger than zero. And indeed it was conjectured in \cite{HolzOrlZee} that $\Sigma$ has fractal dimension in the range $(1,2)$ (and so Lebesgue measure zero). That this is not the case was shown in \cite{CWChonchaiyaLindner2011} by an application of Theorem \ref{thm_lo2}, specifically by constructing a sequence $b\in \{\pm 1\}^\Z$ for which $\sppt A^b \supset \D$, the open unit disc. Of course, this implies by Theorem \ref{thm_lo2} the following result.

\noindent
\ifpics{  
\begin{center}
\begin{tabular}{ccc}
\includegraphics[width=0.3\textwidth]{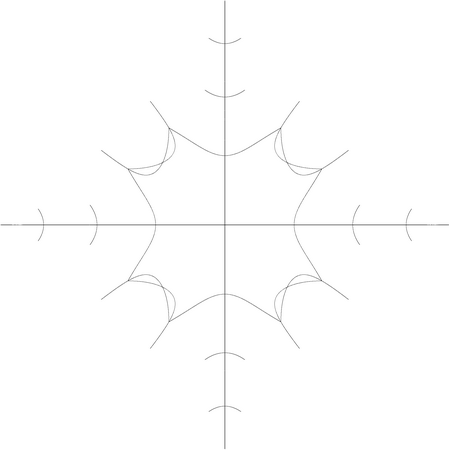}&
\includegraphics[width=0.3\textwidth]{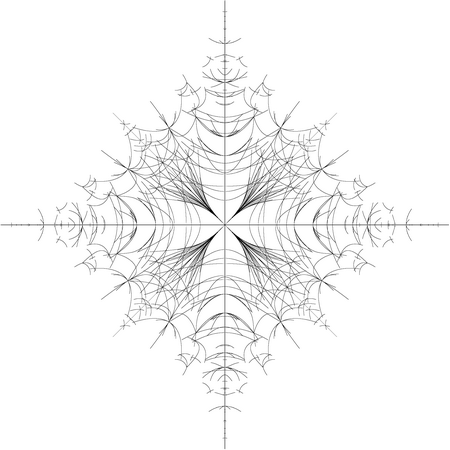}&
\includegraphics[width=0.3\textwidth]{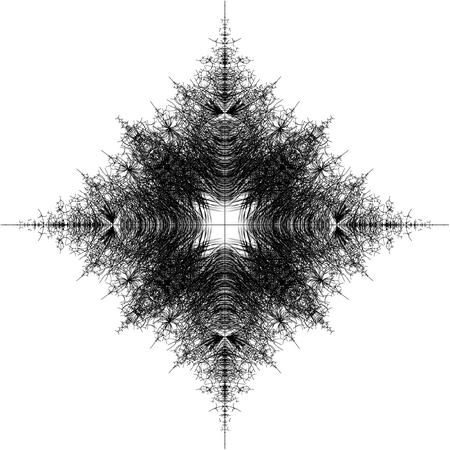}\\
\includegraphics[width=0.3\textwidth]{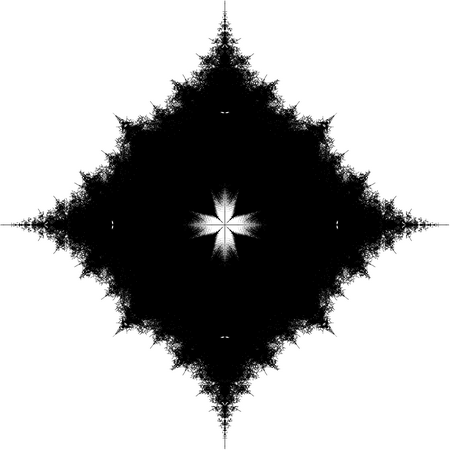}&
\includegraphics[width=0.3\textwidth]{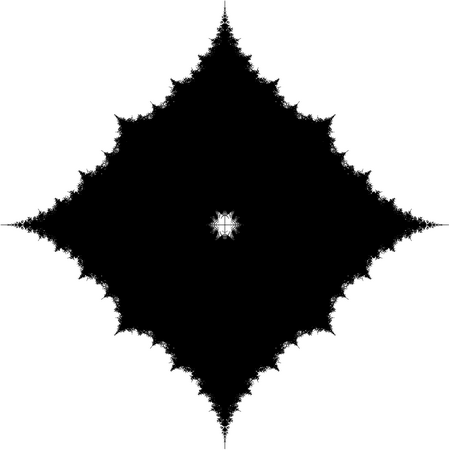}&
\includegraphics[width=0.3\textwidth]{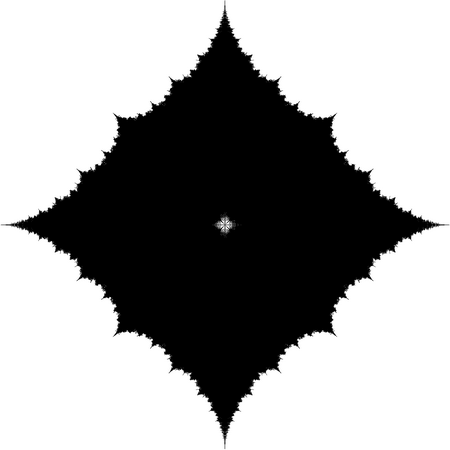}
\end{tabular}
\end{center}
}\fi 
\begin{figure}[h]
\caption{\footnotesize Our figure shows the sets $\pi_n$, as defined
in (\ref{eq:spperntri}), for $n=5,10,...,30$, computed using the characterisation (\ref{perspec}), which is made explicit for $n=1,2$ and 3 in Lemma \ref{periods123}. In particular, $\pi_1 = [-2,2]\cup \ri[-2,2]$ and, for each $n$, $\pi_1\subset \pi_n$ and, by Lemma \ref{lem_nr}, $\pi_n\subset \overline{\Delta} = \{x+\ri y: x,y\in \R, \, |x|+|y| \leq 2\}$.} \label{fig:30pics1}
\end{figure}

\begin{theorem}\cite[Proposition 2.1]{CWChonchaiyaLindner2011} \label{thm_sier} $\overline{\D} \subset \Sigma$.
\end{theorem}

\noindent Recently \cite{CWDavies2011}, an alternative proof of this theorem has been obtained, through a construction that shows that $\pi_\infty$ is dense in $\D$. It is an open (and interesting) question as to whether $\pi_\infty$ is dense in $\Sigma$. An interesting, related, case where the union of the spectra of all periodic operators is shown to be dense in the spectrum of the pseudo-ergodic case is studied in \cite{LiBiDiag}, but there are other pseudo-ergodic bi-infinite tridiagonal examples where this is not true.

The above results concern bi-infinite matrices, but similar results apply to the semi-infinite matrices $A^b_+$ and $A^{b,c}_+$. We say that the operator induced by the bi-infinite matrix
$B=(b_{ij})_{i,j\in\N}$ is a {\sl limit operator} of the operator induced
by the banded semi-infinite matrix $A_+=(a_{ij})_{i,j\in\N}$ if, for a sequence
$h_1,h_2,...$ of integers with $h_k\to+\infty$, it holds that
\[
a_{i+h_k,j+h_k}\ \to\ b_{ij}\qquad\textrm{as}\qquad k\to\infty,
\]
for all $i,j\in\Z$. The set of all limit operators of $A_+$ is denoted by
$\opsp(A_+)$. An equivalent characterisation is that $\opsp(A_+) = \opspp(\tilde A_+)$, where, for any semi-infinite matrix $A_+$,  $\tilde A_+$ is the bi-infinite matrix defined by $\tilde A_+ = (\tilde a_{ij})_{i,j\in\Z}$, where $\tilde a_{ij} := a_{ij}$, $i,j\in\N$, $\tilde a_{ij}:=0$, otherwise. The following version of Theorem \ref{thm_lo} holds in the semi-infinite case. In its results on the pseudospectrum this theorem appears to be new and may be of independent interest. The arguments in this theorem and in later sections depend on the following lemma which, in its results for the pseudospectrum, generalises \cite[Theorem 2.4(iii)]{TrefEmbBook} from the finite-dimensional Hilbert space case to an infinite-dimensional Banach space setting, and so may also be of independent interest.

\begin{lemma} \label{lem_sub} Suppose that $X$ is a Banach space which can be written as the direct sum of two closed subspaces as $X = X_1\oplus X_2$, by which we mean that each $x\in X$ can be written in a unique way as $x=x_1+x_2$ with $x_1\in X_1$ and $x_2\in X_2$, and that there exists a continuous projection operator $P_1:X\to X_1$ (in which case $P_2=I-P_1$ is a projection operator onto $X_2$). Suppose also that $A$ is a bounded linear operator on $X$ which has $X_1$ and $X_2$ as invariant subspaces, and let $A_j$ denote $A$ restricted to $X_j$, for $j=1,2$. Then
$\spec A = \spec A_1 \cup \spec A_2$, $\spess A = \spess A_1 \cup \spess A_2$, and $\specn_\eps A_j \subset \specn_\eps A$, for $\eps>0$, and $j=1,2$.
If, for some $p\in [1,\infty]$, it holds for every $x_1\in X_1$ and $x_2\in X_2$ that $\|x_1+x_2\| = \|(\|x_1\|,\|x_2\|)\|_p$,  then also
$\specn_\eps A = \specn_\eps A_1 \cup \specn_\eps A_2$, for $\eps>0$.
\end{lemma}
\begin{proof} The identities
$\spec A = \spec A_1 \cup \spec A_2$ and $\spess A = \spess A_1 \cup \spess A_2$ are standard, see e.g.\ \cite{Jorgens_book,Davies2007:Book}. By Theorem \ref{thm_pse},  $\specn_\eps B = \spec B \cup \{\lambda\in \C: \nu(B-\lambda I)<\eps\}$. Since $\nu(A_j-\lambda I)\geq \nu(A-\lambda I)$, for all $\lambda\in\C$ and $j=1,2$, it follows that $\specn_\eps A_j \subset \specn_\eps A$, for $\eps>0$, and $j=1,2$. If, for some $p\in [1,\infty]$, it holds for every $x_1\in X_1$ and $x_2\in X_2$ that $\|x_1+x_2\| = \|(\|x_1\|,\|x_2\|)\|_p$,  then, for every $\lambda\in\C$, where $B:= A-\lambda I$ and $B_j := A_j - \lambda I$, for $j=1,2$, it holds for $x_1\in X_1$ and $x_2\in X_2$ that
$$
\|B(x_1+x_2)\| =  \big\|\big(\|Bx_1\|,\|Bx_2\|\big)\big\|_p\geq \big\|\big(\nu(B_1)\|x_1\|,\nu(B_2)\|x_2\|\big)\big\|_p
$$
so that
$$
\nu(B) = \inf_{x_1\in X_1, x_2\in X_2} \frac{\|B(x_1+x_2)\|}{\|x_1+x_2\|}\geq \inf_{x_1\in X_1, x_2\in X_2} \frac{\big\|\big(\nu(B_1)\|x_1\|,\nu(B_2)\|x_2\|\big)\big\|_p}{\big\|\big(\|x_1\|,\|x_2\|\big)\big\|_p}.
$$
But it is an easy calculation that this last infimum has the value $\min(\nu(B_1), \nu(B_2))$. Thus $\specn_\eps A \subset \specn_\eps A_1 \cup \specn_\eps A_2$.
\end{proof}

\begin{theorem} \label{thm_losemi} Let $A_+$ be a semi-infinite banded matrix $A_+=(a_{ij})_{i,j\in\N}$, with $\sup_{ij}|a_{ij}|<\infty$. Then
\begin{equation} \label{eq:lo7}
\spess A_+ = \bigcup_{B\in \opsp(A_+)} \spec B = \bigcup_{B\in \opsp(A_+)} \sppt B.
\end{equation}
Further, $\specn_\eps^p B \subset \specn_\eps^p A_+$, for all $\eps>0$, $p\in [1,\infty]$, and $B\in \opsp(A_+)$.
\end{theorem}
\begin{proof}
 Given $\eps>0$ and $p\in [1,\infty]$, choose  $\lambda> \|A_+\|_p + 2\eps$ and apply Theorem \ref{thm_lo} to the bi-infinite matrix $A = \tilde A_+ + B$, where $B=(b_{ij})_{i,j\in\Z}$ is defined by $b_{ij}:=\lambda$, if $i=j\leq 0$, $b_{ij}:=0$, otherwise. Since $\opsp(A)= \opspp(A)\cup \opspm(A) = \opsp(A_+)\cup \{\lambda I\}$, we see, applying Lemma \ref{lem_sub}, that
\begin{equation} \label{eq:lo3}
\{\lambda\} \cup \spess A_+ = \spess A = \{\lambda\}\cup \bigcup_{B\in \opsp(A_+)} \spec B = \{\lambda\}\cup\bigcup_{B\in \opsp(A_+)} \sppt B.
\end{equation}
Since $\lambda>\|A_+\|_p\geq \|B\|_p$ is not in $\spess A_+$ or in $\spec B$, for $B\in \opsp(A_+)$, equation (\ref{eq:lo7}) follows. Similarly, applying Lemma \ref{lem_sub}, $\specn_\eps^p A = \specn_\eps^p(\lambda I) \cup \specn_\eps^p A_+ = (\lambda+\eps\D) \cup \specn_\eps^p A_+$. It follows from Theorem \ref{thm_lo} that, for $B\in \opsp(A_+)$, $\specn_\eps^p B \subset \specn_\eps^p A = (\lambda+\eps\D) \cup \specn_\eps^p A_+$. Since $\specn_\eps^p B \subset (\|B\|_p + \eps)\D$ and $\lambda > \|A_+\|_p + 2\eps \geq \|B\|_p + 2\eps$, this implies that $\specn_\eps^p B \subset \specn_\eps^p A_+$.
\end{proof}

One consequence of this result and Theorem \ref{thm_lo} is the following lemma.

\begin{lemma} \label{lem_zero} For $b,c\in\{\pm 1\}^\Z$, $0\in \spess A^{b,c}$ and $0\in \spess A_+^{b,c}$.
\end{lemma}
\begin{proof}
It is easy to see that $0\in \sppt A^{b,c}$ for every $b,c\in \{\pm 1\}^\Z$, and the result then follows from equations (\ref{eq:lo1}) and (\ref{eq:lo7}).
\end{proof}

We extend the definition of pseudo-ergodic in Definition \ref{pse} to the semi-infinite case in the obvious way, replacing $\Z$ by $\N$ and $A^b$ by $A^b_+$, so that $b\in \{\pm 1\}^\N$ is pseudo-ergodic iff the sequence $b$ contains every finite pattern of $\pm 1$'s. Then Lemma \ref{imt}  holds with $\Z$ replaced by $\N$ and $A^b$ replaced by $A^b_+$, and Lemma \ref{pe_lo} holds with  $A^b$ replaced by $A^b_+$.

\section{The Numerical Range and Symmetry Arguments} \label{sec:nrsymm}

  Let us first introduce some properties of and notation related to adjoint operators. Given a banded bi-infinite matrix  $A=(a_{ij})_{i,j\in\Z}$, with $\sup_{ij} |a_{ij}|<\infty$, $A^*$ will denote the matrix $A^*=(\bar a_{ji})_{i,j\in \Z}$. For $1\leq p<\infty$, where $q\in (1,\infty]$ satisfies $p^{-1}+q^{-1}=1$, and identifying $\ell^q(\Z)$ with $(\ell^p(\Z))^*$, the dual space of $\ell^p(\Z)$ (in the sense e.g.\ of Kato \cite{Kato_book}, where the elements of the dual space are anti-linear functionals), it holds that  $A^*:\ell^q(\Z)\to\ell^q(\Z)$ is the adjoint of $A:\ell^p(\Z)\to\ell^p(Z)$. Further \cite{Kato_book} $A$ is invertible iff $A^*$ is invertible and, if they are both invertible, then $\|A^{-1}\|_p=\|(A^*)^{-1}\|_q$. Similarly, $A$ is Fredholm iff $A^*$ is Fredholm and, if they are both Fredholm then $\ind A = - \ind A^*$.


In this section we first compute the numerical range of the operator $A^b$ in the case when $b$ is pseudo-ergodic, which gives an upper bound on the spectrum $\Sigma$ of $A^b$. We then apply a variety of symmetry arguments to explore the relationship between spectral sets for matrices with one and two $\pm 1$ diagonals and between semi-infinite and bi-infinite matrices, and to explore the geometry of $\Sigma$ and that of $\Sigma_\eps^p$, the $\eps$-pseudospectrum of $A^b$ on $\ell^p(\Z)$ when $b$ is pseudo-ergodic. Our final result shows that, roughly speaking, in the pseudo-ergodic case, the spectral sets are the same whether the matrix is semi-infinite or bi-infinite, and whether the matrix has one or two $\pm 1$ diagonals.

These results are to some extent surprising: there is no expectation in general that the spectral sets associated with bi-infinite and corresponding semi-infinite matrices will be the same. A simple example is provided by the shift operator $V_{-1}$. This is a Laurent operator (a bi-infinite Toeplitz matrix) whose spectrum is the unit circle and whose $\ell^2$ $\eps$-pseudospectrum is the $\eps$-neighbourhood of the unit circle. On the other hand the Toeplitz operator that is the shift operator restricted to $\ell^2(\N)$ (a semi-infinite Toeplitz matrix) has spectrum that is the closed unit disc (e.g.\ \cite{Davies2007:Book}). An example closer to our case is studied in \cite{TrefContEmb}, where calculations are made of the spectra of random bi-diagonal bi-infinite and semi-infinite matrices, matrices which the authors term stochastic Laurent and Toeplitz operators, respectively, by which they mean a bi-infinite or semi-infinite matrix where each diagonal is either constant or has random entries, but with the random distribution constant along the diagonal. In the bi-diagonal case they study, which has the constant value 1 along the first superdiagonal and a random main diagonal, it is found \cite{TrefContEmb} that the bi-infinite and semi-infinite matrices may or may not have the same spectra, this depending on the support of the probability density function for the random variables on the main diagonal.

Our first result is a computation of the numerical range. By $W(B)$ we denote the (2-norm) {\em numerical range} of the operator or matrix $B$, defined by (see Section \ref{sec:pseudo})
$
W(B) := \{(Bx,x):\|x\|_2=1\}$,
where $(\cdot,\cdot)$ denotes the standard $\ell^2$ inner-product on $\C^n$ or on $\ell^2(S)$, with $S=\Z$ or $\N$, as appropriate.

\begin{lemma} \label{lem_nr} For $b\in\{\pm1\}^\Z$, $W(A^b_n) \subset W(A^b_+) \subset W(A^b)\subset \Delta:=\{x+\ri y : x,y\in\R,\, |x|+|y|<2 \}$, and $W(A^b)= \Delta$ if $b$ is pseudo-ergodic. Similarly, $W(A^b_+)= \Delta$ if $b\in\{\pm1\}^\N$ is pseudo-ergodic, and  $\Sigma \subset \overline{\Delta}$.
\end{lemma}
\begin{proof} For $b\in \{\pm 1\}^\Z$ and $x\in \ell^2(\Z)$ with $\|x\|_2=1$, defining $a = \bar xV_{-1} x$ we see that
$$
(A^bx,x) = \sum_{k\in\Z}(b_{k-1} x_{k-1} + x_{k+1})\bar x_k = \sum_{k\in\Z}(b_{k} \bar a_k +  a_k) = \sum_{k\in\Z}[\alpha_k(1+b_{k})+\ri\beta_k(1-b_{k})],
$$
where $\alpha_k = \Re(a_k)$ and $\beta_k = \Im(a_k)$. Thus
$$
|\Re(A^bx,x)| + |\Im(A^bx,x)| \leq \sum_{k\in\Z}\left\{|\alpha_k|(1+b_{k})+ |\beta_k|(1-b_{k})\right\} \leq 2\sum_{k\in \Z} |a_k| = 2 \|a\|_1.
$$
Now, since $x\in \ell^2(\Z)$, $\bar x$ and $V_{-1}x$ must be linearly independent. Hence, by the Cauchy-Schwarz inequality, $\|a\|_1 = \|\bar x V_{-1} x\|_1 < \|x\|_2\|V_{-1}x\|_2 = 1$. We have shown that $W(A^b)\subset \Delta$; it follows that $W(A^b_n)\subset W(A^b_+)\subset W(A^b)$ from (\ref{eq:numr}). From this it follows, from standard properties of the numerical range (see the end of Section \ref{sec:pseudo}) that $\Sigma \subset \overline{W(A^b)} \subset \overline{\Delta}$. But, since $\pi_1 = [-2,2]\cup \ri[-2,2] \subset \Sigma$, this implies that $\pm 2$ and $\pm 2\ri$ are in $\overline{W(A^b)}$, if $b$ is pseudo-ergodic. Hence, if $b$ is pseudo-ergodic, then, for every $\eta>0$ there exist points $r,s,t,u\in W(A^b)$ with $|2-r|<\eta$, $|-2-s|<\eta$, $|2\ri-t|<\eta$ and $|-2\ri-u|<\eta$. Since $W(A^b)$ is convex, this implies that $\Delta\subset W(A^b)$, and so $W(A^b)=\Delta$. A similar argument, using that $\Sigma= \spec A^b_+ \subset \overline{W(A^b_+)}$ if $b\in \{\pm1\}^\N$ is pseudo-ergodic (that $\Sigma = \spec A^b_+$ is established in Theorem \ref{thm_main} below), shows that $W(A^b_+)=\Delta$ if $b\in \{\pm1\}^\N$ is pseudo-ergodic.
\end{proof}

Our next result elucidates the relationship between the spectral properties of matrices with one and two $\pm 1$ diagonals. One obvious symmetry result we use already in this lemma is that, since the coefficients $b,c\in \{\pm 1\}^\Z$ are real-valued, the spectrum and pseudospectrum of $A^{b,c}$ are symmetric about the real axis.

\begin{lemma} \label{lem_sim}
For $a, b,c\in \{\pm1\}^\Z$,
$$
M_a A^{b,c} M_a^{-1}= A^{bd,cd},
$$
where $d = aV_{-1} a$, so that
$$
\spec A^{b,c} = \spec A^{bd,cd} = \spec A^{bc}, \quad \spess A^{b,c} = \spess A^{bd,cd} = \spess A^{bc}.
$$
Further, for $\lambda \not\in \spess A^{b,c}$,
$\ind (A^{b,c}-\lambda I) = 0
$,
and, for $\lambda \not\in \spec A^{b,c}$ and $p\in [1,\infty]$, where $q\in [1,\infty]$ is given by $p^{-1}+q^{-1}=1$,
$$
\|(A^{b,c}-\lambda
I)^{-1}\|_p=\|(A^{bd,cd}-\lambda I)^{-1}\|_p=\|(A^{bc}-\lambda I)^{-1}\|_p=\|(A^{b,c}-\lambda
I)^{-1}\|_q,
$$
so that, for $\eps >0$,
$$
\specn_\eps^p A^{b,c} = \specn_\eps^p A^{bd,cd} = \specn_\eps^p A^{bc} = \specn_\eps^q A^{b,c}.
$$
Moreover, for $1\leq p \leq r\leq 2$ and $\eps>0$,
$$
\specn_\eps^r A^{b,c} \subset \specn_\eps^p A^{b,c}.
$$
\end{lemma}
\begin{proof} For $a,b,c\in \{\pm 1\}^\Z$, recalling (\ref{eq:Abrep}) and noting that $M_a^{-1}=M_a$,
$$
M_a A^{b,c} M_a^{-1}= M_{a}V_1M_{ab} + M_{ac} V_{-1} M_{a} = V_1 M_{V_{-1}a} M_{ab} + M_{ac}M_{V_{-1} a}V_{-1} =  V_1 M_{bd}+ M_{cd}V_{-1} =A^{bd,cd}.
$$
In particular, choosing $a$ so that $d=c$, this identity reduces to $M_a A^{b,c} M_a^{-1}= A^{bc}$, while, choosing $a$ so that $d=bc$, this identity reduces to $M_a A^{b,c} M_a^{-1}= A^{c,b}=(A^{b,c})^*$.
The remaining results, except the last equation, follow since $M_a$ is an isometric isomorphism, and using the properties of the adjoint listed immediately at the beginning of the section, and standard properties of Fredholm operators, e.g.\ \cite{Kato_book,Jorgens_book}. The last inclusion follows from the interpolation theorem of Riesz-Thorin, often called the Riesz convexity theorem \cite[Chapter V,
Theorem 1.3]{SteinWeissBook}, which implies that, for $\lambda\not \in \spec A^{b,c}$,
$$
\|(A^{b,c}-\lambda I)^{-1}\|_r \leq \max(\|(A^{b,c}-\lambda I)^{-1}\|_p,\|(A^{b,c}-\lambda I)^{-1}\|_q) = \|(A^{b,c}-\lambda I)^{-1}\|_p.
$$
\end{proof}

Note that this lemma implies that, for $1\leq p \leq 2\leq q\leq \infty$, where $p^{-1}+q^{-1}=1$,
$$
\specn_\eps^2 A^{b,c} \subset \specn_\eps^p A^{b,c} = \specn_\eps^q A^{b,c} \subset \specn_\eps^1 A^{b,c} = \specn_\eps^\infty A^{b,c}.
$$
In general, for a non-self-adjoint operator or matrix $A$, it need not hold that $\specn_\eps^r A\subset \specn_\eps^p A$ for any distinct $p,r\in [1,\infty]$.

Exactly the same results hold in the semi-infinite case. Precisely, where $M^+_a$ denotes the operator on $\ell^p(\N)$ of multiplication by $a\in \{\pm 1\}^\N$, Lemma \ref{lem_sim} holds also with $\Z$ replaced by $\N$, $M_a$ replaced by $M_a^+$, and all other operators replaced by their semi-infinite counterparts. Similarly, where $D^a_n$ is the diagonal matrix with the vector $a = (a_1,...,a_n)$ on the diagonal, the following finite dimensional version of the above lemma holds.

\begin{lemma} \label{lem_simfd}
For $n\in \N$, $a\in \{\pm 1\}^n$, and $b,c\in \{\pm 1\}^{n-1}$,
$$
D^a_n A_n^{b,c} D^a_n= A_n^{bd,cd},
$$
where $d = (a_1a_2,...,a_{n-1}a_n)$, so that
$$
\spec A_n^{b,c} = \spec A_n^{bd,cd} = \spec A_n^{bc}.
$$
Further, for $p\in [1,\infty]$ and $\eps >0$,  where $q\in [1,\infty]$ is given by $p^{-1}+q^{-1}=1$,
$$
\specn_\eps^p A_n^{b,c} = \specn_\eps^p A_n^{bd,cd} = \specn_\eps^p A_n^{bc} = \specn_\eps^q A_n^{b,c}.
$$
Moreover, for $1\leq p \leq r\leq 2$ and $\eps>0$,
$
\specn_\eps^r A_n^{b,c} \subset \specn_\eps^p A_n^{b,c}$.
\end{lemma}

A first application of the above lemmas is the following symmetry result (cf.\ \cite{HolzOrlZee}).
\begin{lemma} \label{lem_symm} For $b\in \{\pm 1\}^\Z$, $\eps>0$, and $p\in [1,\infty]$, $\spec A^{b}$, $\spess A^{b}$, $\specn^p_\eps A^{b}$, $\spec A_n^{b}$, and $\specn_\eps^p A_n^{b}$  are invariant under reflection in the real and imaginary axes. Further, where $S(b)$ denotes any one of these sets, it holds that $S(-b) = \ri S(b)$. The set  $\Sigma$, which is the set $\spec A^b = \spess A^b$ in the case that $b$ is pseudo-ergodic, and, for $\eps>0$ and $p\in[1,\infty]$, the set $\Sigma_\eps^p$, which is the set $\specn^p_\eps A^{b}$ for $b$ pseudo-ergodic, are invariant under reflection in either axis and under rotation by $90^0$.
\end{lemma}
\begin{proof} We prove the results for $A^{b}$ using Lemma \ref{lem_sim}; the proof for $A_n^{b}$ using Lemma \ref{lem_simfd} is similar. That the entries of the matrix $A^b$ are real implies the symmetry about the real axis. Defining $a\in \{\pm1\}^\Z$ by $a_k = (-1)^k$, $k\in\Z$, so that $d = a V_{-1} a$ is the constant sequence $d = (...,-1,-1,...)$, it follows from Lemma \ref{lem_sim} that $M_a A^b M_a^{-1} = -A^b$, which implies that the sets  $\spec A^{b}$, $\spess A^{b}$, and $\specn^p_\eps A^{b}$ are also invariant under reflection in the origin, so that they are also invariant under reflection in the imaginary axis. Defining, instead, $a\in \ell^\infty(\Z)$ by $a_k = \ri^k$, we obtain, similarly, that $M_a A^b M_a^{-1} = A^{db,\bar d}$, where $d = \bar a V_{-1} a$ so that $d_k = \ri$. Thus $M_a A^b M_a^{-1} = -\ri A^{-b}$, and we see that $S(-b)=\ri S(b)$, where $S(b)$ denotes one of $\spec A^{b}$, $\spess A^{b}$, or $\specn^p_\eps A^{b}$. Where $S(b)$ again denotes one of these sets, since $b$ is pseudo-ergodic iff $-b$ is pseudo-ergodic, that $S(b)=S(-b)=\ri S(b)$ follows from Theorem \ref{thm_lo2}.
\end{proof}

The following lemma further elucidates the relationship between the spectral properties of semi-infinite and bi-infinite matrices. In this lemma for $p\in[1,\infty]$ we let $\ell_o^p(\Z)$ denote the closed subspace of odd elements of $\ell^p(\Z)$, i.e.\ $x\in \ell^p_o(\Z)$ iff $x_{-k}=-x_k$, $k\in\Z$, and let $\ell_e^p(\Z)$ denote the closed subspace of even elements of $\ell^p(\Z)$, i.e.\ $x\in \ell^p_e(\Z)$ iff $x_{-k}=x_k$, $k\in\Z$, so that $\ell^p(\Z)=\ell^p_o(\Z)\oplus \ell^p_e(\Z)$. It is convenient to equip $\ell^p_o(\Z)$ with the norm $\|x\| := 2^{-1/p} \|x\|_p$, so that the extension operator $E:\ell^p(\N)\to \ell^p_o(\Z)$ given by $(Ex)_k = x_k$, $k\in\N$, $(Ex)_0=0$, and $(Ex)_{-k} = -x_k$, $k\in\N$, is an isometric isomorphism, as is the restriction operator $P:\ell^p_o(\Z)\to \ell^p(\N)$ given by $(Px)_k = x_k$, $k\in\N$. (This change of norm does not effect the value of the induced norm of a bounded linear operator $A$ on $\ell^2_o(\Z)$, and so does not affect the definition of $\specn_\eps^p A$.) Further, let $R:\ell^p(\Z)\to \ell^p(\Z)$ be the reflection operator given by $(Rx)_k = x_{-1-k}$, $k\in\Z$.

\begin{lemma} \label{lem_ref}
Suppose $b\in \{\pm 1\}^\Z$ with $b_k = 1$, $k\leq 0$, and let $c = Rb$. Then, for $p\in [1,\infty]$, $A^{b,c}$ maps $\ell^p_o(\Z)$ to $\ell^p_o(\Z)$ and maps $\ell^p_e(\Z)$ to $\ell^p_e(\Z)$. Further, where $A^{b,c}_o$ denotes the restriction of $A^{b,c}$ to $\ell^p_o(\Z)$,
\begin{equation} \label{eq_semi}
A_o^{b,c} = E A_+^b P.
\end{equation}
Thus $\spec A_+^b = \spec A_o^{b,c} \subset \spec A^{b,c} = \spec A^{bc}$ and $\specn_\eps^p A_+ = \specn_\eps^p A_o^{b,c}\subset \specn_\eps^p A^{b,c}=\specn_\eps^p A^{bc}$, for $\eps>0$ and $p\in [1,\infty]$.
\end{lemma}
\begin{proof} For $x\in \ell^2_o(\Z)$, $(A^{b,c}x)_0 = b_{-1}x_{-1}+c_0 x_1 = x_{-1} + x_1=0$ and, for $k\in\N$, $(A^{b,c}x)_{-k} = b_{-k-1}x_{-k-1}+c_{-k} x_{-k+1}= -c_{k}x_{k+1}-b_{k-1} x_{k-1}=-(A^{b,c}x)_k$, so that $A^{b,c}:\ell^p_o(\Z)\to\ell^p_o(\Z)$. Similarly, for $x\in \ell^2_e(\Z)$ and $k\in\N$, $(A^{b,c}x)_{-k} = b_{-k-1}x_{-k-1}+c_{-k} x_{-k+1}= c_{k}x_{k+1}+b_{k-1} x_{k-1}=(A^{b,c}x)_k$, so that $A^{b,c}:\ell^p_e(\Z)\to\ell^p_e(\Z)$. Further, for $k\in \N$, $(EA_+^bx)_k = b_{k-1}x_{k-1}+x_{k+1}= b_{k-1}x_{k-1} + c_kx_{k+1}=(A_o^{b,c}x)_k$, so that (\ref{eq_semi}) holds. Since $E$ and $P$ are isometric isomorphisms and $E=P^{-1}$, it follows that $\spec A_+^b = \spec A_o^{b,c}$ and that $\specn_\eps^p A_+ = \specn_\eps^p A_o^{b,c}$, for $\eps>0$ and $p\in [1,\infty]$. The remaining results follow from Lemma \ref{lem_sub} and Lemma \ref{lem_sim}.
\end{proof}

Putting the results from the previous section and this section together gives the following characterisations of the spectrum, essential spectrum, and pseudospectrum in the pseudo-ergodic case.

\begin{theorem} \label{thm_main} If $b,c,d\in \{\pm 1\}^\N$, $e,f,g\in \{\pm 1\}^\Z$, and $b$, $cd$, $e$, and $fg$ are pseudo-ergodic, then $\spec A^b_+=\spec A^{c,d}_+=\spec A^{e}=\spec A^{f,g}=\spess A^b_+=\spess A^{c,d}_+=\spess A^{e}=\spess A^{f,g}=\Sigma$ and, for $\eps>0$ and $p\in [1,\infty]$, where $q\in [1,\infty]$ is given by $p^{-1}+q^{-1}=1$, $\specn_\eps^p A^b_+=\specn_\eps^p A^{c,d}_+=\specn_\eps^p A^{e}=\specn_\eps^p A^{f,g}=\Sigma_\eps^p=\Sigma_\eps^q$. Further, for  $1\leq p \leq r\leq 2$ and $\eps>0$,
$\Sigma_\eps^r\subset \Sigma_\eps^p$.
\end{theorem}
\begin{proof} From  Lemma \ref{lem_sim} and the remarks following that lemma we have that $\spec A^{c,d}_+=\spec A_+^{cd}$, $\spess A_+^{c,d}=\spess A_+^{cd}$, $\spec A^{f,g}=\spec A^{fg}$, and $\spess A^{f,g}=\spess A^{fg}$. From Theorems \ref{thm_lo2}, \ref{thm_losemi}, and the remarks at the end of Section \ref{sec:lo}, we have moreover that if $b$ and $e$ are pseudo-ergodic then $\spec A^e = \spess A^e = \spess A^b_+ = \Sigma$. This implies that $\Sigma \subset \spec A^b_+$, and that $\spec A^b_+ \subset \Sigma$ follows from Lemma \ref{lem_ref} which, together with Lemma \ref{lem_sim}, gives that $\spec A^b_+ \subset \spec A^{b,Rb} = \spec A^{bRb} \subset \Sigma$.
 The results for the pseudospectrum are shown similarly, again using Theorems \ref{thm_lo2}, \ref{thm_losemi}, the remarks at the end of Section \ref{sec:lo}, and Lemma \ref{lem_sim}.
\end{proof}

\section{The relationship between the spectra and pseudospectra of finite and infinite matrices} \label{sec:fininf}
An obvious method to try to calculate the spectrum of an infinite matrix is to study the spectra of large finite submatrices of the infinite matrix and hope that these provide good approximations. In particular, one can apply this idea to the infinite matrix $A^b_+$, and hope that the spectrum of the $n\times n$ matrix $A^b_n$, the intersection of the first $n$ rows and columns of $A^b_+$, will approximate the spectrum of $A^b_+$ well for $n$ large.

In general  the spectrum and pseudospectrum of an infinite banded matrix may or may not be well-approximated by the spectra and pseudospectra of its finite submatrices (see \cite{LindnerRoch2010} and the references therein for some discussion, with emphasis on the case of tridiagonal pseudo-ergodic matrices). In particular, there need be no relationship at all between the spectrum of a bi-infinite matrix and the spectra of its finite sections. A simple example is provided by the Laurent operator that is the shift operator $V_{-1}$ with matrix representation $(a_{ij})_{ij\in\Z}$, with $a_{ij}=\delta_{i,j+1}$, whose spectrum is the unit circle. The Toeplitz matrices that are its $n\times n$ finite sections, $(a_{ij})_{1\leq i,j\leq n}$, clearly have zero as the only eigenvalue.

The purpose of this section is to show that, for the particular class of pseudo-ergodic operators we are studying, there is a perhaps surprisingly close (given that our pseudo-ergodic operators are not self-adjoint or normal) connection between the spectral sets in the finite and infinite case. This connection is particularly close for the pseudospectra.

\subsection{That the finite matrix spectral sets are contained in the infinite matrix counterparts}

For $n\in\N$, introduce the $n\times n$
matrices
\[
I_n\ =\ \left(\begin{array}{ccc} 1\\ &\smash{\ddots}\\& &1\end{array}\right)
\qquad\textrm{and}\qquad J_n\ =\ \left(\begin{array}{ccc}& &1\\ &\rdots\\
1\end{array}\right),
\]
so that $I_n$ is the order $n$ identity matrix. The proof of the following result uses a similar construction to that of the bi-infinite matrix $A^{b,c}$ in the proof of Lemma \ref{lem_ref}.

\begin{theorem} \label{prop:spec_An}
If $b$ is pseudo-ergodic then, for $n\in\N$,
$$
\spec A_{n}^{b}\subset \sigma_n := \bigcup_{f\in\{\pm 1\}^{n-1}} \spec A_n^f\  \subset \pi_{2n+2}\subset
\spec A^{b} = \Sigma.
$$
\end{theorem}

\noindent
\ifpics{ 
\begin{center}
\begin{tabular}{ccc}
\includegraphics[width=0.3\textwidth]{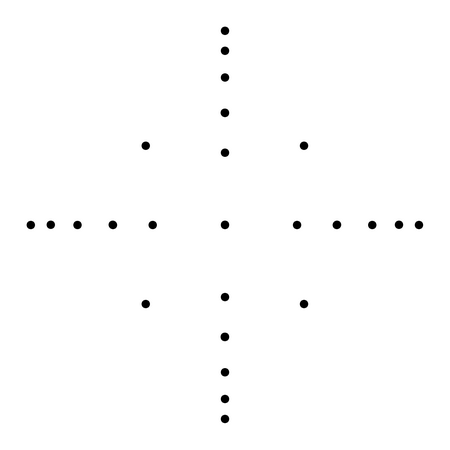}&
\includegraphics[width=0.3\textwidth]{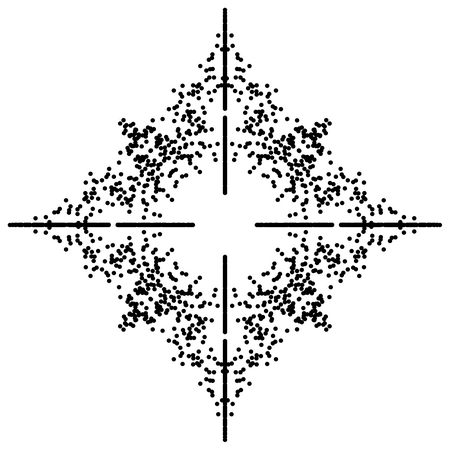}&
\includegraphics[width=0.3\textwidth]{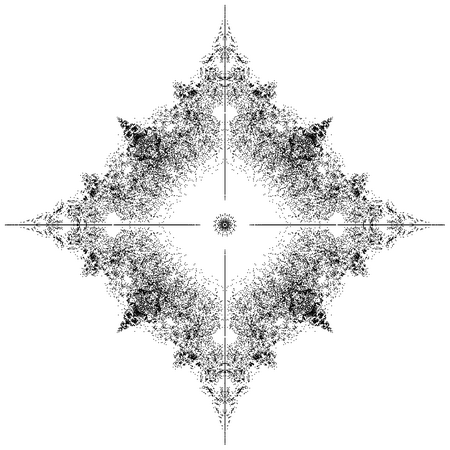}\\
\includegraphics[width=0.3\textwidth]{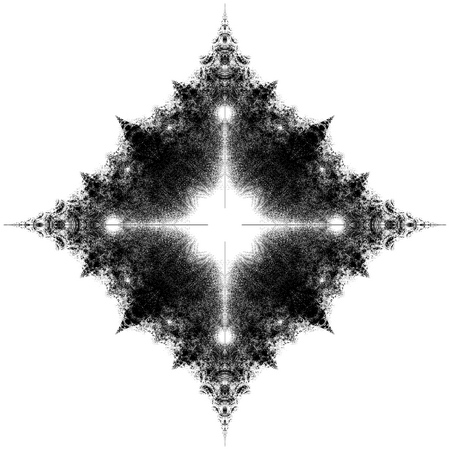}&
\includegraphics[width=0.3\textwidth]{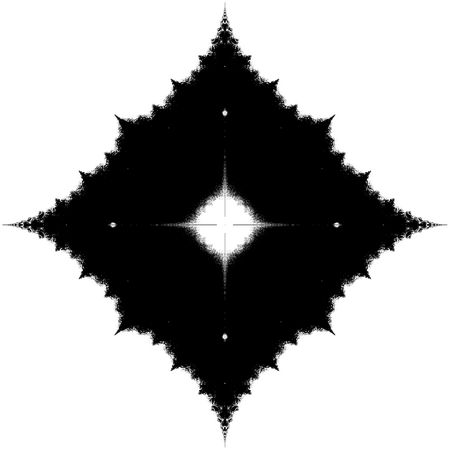}&
\includegraphics[width=0.3\textwidth]{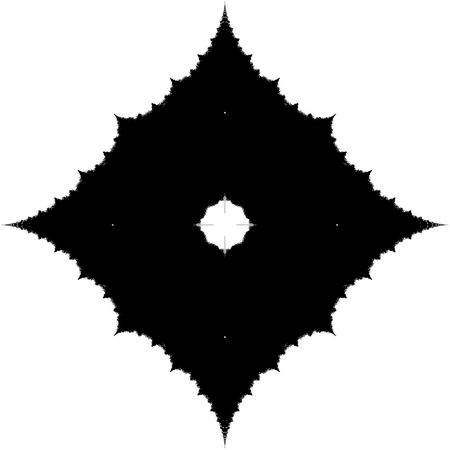}
\end{tabular}
\end{center}
}\fi 
\begin{figure}[h]
\caption{\footnotesize Our figure shows the sets $\sigma_n$ of all $n\times n$ matrix eigenvalues, as defined in
Theorem \ref{prop:spec_An}, for $n=5,10,...,30$. Note that in the first pictures
(with only a few eigenvalues), we have used heavier pixels for the
sake of visibility.}  \label{fig:30pics2}
\end{figure}
\begin{proof}
If $\lambda\in \spec A_{n}^{f}$, for some $f\in \{\pm 1\}^{n-1}$, then $A_{n}^{f}x=\lambda x$ for
some non-zero $x\in\C^n$. Put $\widehat x:= J_{n}x$ and
$\widehat{A^{f}_{n}}:= J_{n}A^{f}_{n}J_{n}$. Then
\[
\widehat{A^{f}_{n}}\widehat{x}\ =\ J_{n}A^{f}_{n}J_{n}J_{n}x\ =\
J_{n}A^{f}_{n}x\ =\ J_{n}\lambda x=\lambda \widehat{x}
\]
and hence, using repeated reflections, i.e. by putting
\[
A^{c,d}:=
\left(\begin{array}{ccccccc}
\ddots&
\begin{array}{ccc}1&& \end{array}\\
\cline{2-2}
\begin{array}{c}1\\ \\ \\ \end{array}&
\multicolumn{1}{|c|}{\widehat{A_n^f}}&
\begin{array}{c} \\ \\1 \end{array}\\
\cline{2-2}
&\begin{array}{ccr}&&-1\end{array}
& 0 & \begin{array}{lcc}1&&\end{array}\\
\cline{4-4}
& & \begin{array}{c}1\\ \\ \\ \end{array}&
\multicolumn{1}{|c|}{{A_n^f}}&
\begin{array}{c} \\ \\1 \end{array}\\
\cline{4-4}
&&&\begin{array}{ccr}&&-1\end{array}
& 0 & \begin{array}{lcc}1&&\end{array}\\
\cline{6-6}
& && & \begin{array}{c}1\\ \\ \\ \end{array}&
\multicolumn{1}{|c|}{\widehat{A_n^f}}&
\begin{array}{c} \\ \\1 \end{array}\\
\cline{6-6}
&&&&&
\begin{array}{ccc}&&-1\end{array}&
\ddots
\end{array}\right)\quad \textrm{and}\quad
\widetilde x:=\left(\begin{array}{c}
\vdots\\
\hline
\multicolumn{1}{|c|}{~}\\
\multicolumn{1}{|c|}{\widehat{x}}\\
\multicolumn{1}{|c|}{~}\\
\hline
0\\
\hline
\multicolumn{1}{|c|}{~}\\
\multicolumn{1}{|c|}{{x}}\\
\multicolumn{1}{|c|}{~}\\
\hline
0\\
\hline
\multicolumn{1}{|c|}{~}\\
\multicolumn{1}{|c|}{\widehat{x}}\\
\multicolumn{1}{|c|}{~}\\
\hline
\vdots
\end{array}
\right),
\]
we get $A^{c,d}\widetilde{x}=\lambda \widetilde{x}$ with
$\widetilde{x}\in\ell^{\infty}(\Z)$, so that $\lambda$ is an eigenvalue
of $A^{c,d}$ as an operator on $\ell^{\infty}(\Z)$. Thus, applying Lemma \ref{lem_sim} and Theorem \ref{thm_lo2}, and noting that both $c$ and $d$ are periodic, with period $2n+2$, we see that
$
\lambda \in \spec A^{c,d} = \spec A^{cd} \subset \pi_{2n+2} \subset \Sigma = \spec A^b.
$
\end{proof}

\begin{center}
\ifpics{\includegraphics[width=0.6\textwidth]{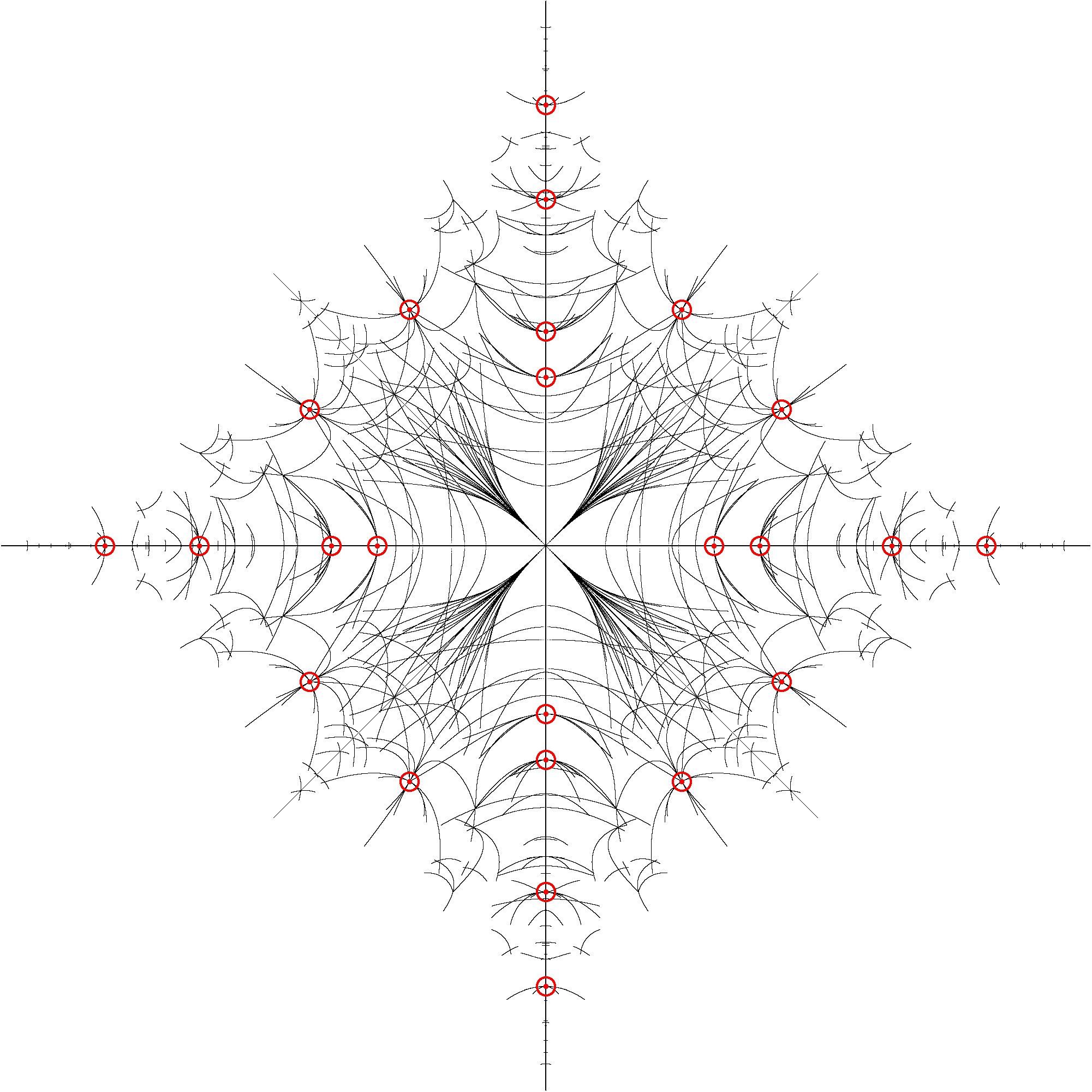}}\fi
\end{center}
\begin{figure}[h]
\caption{\footnotesize An illustration of  the inclusion $\sigma_4\ \subset\
\pi_{10}$, which holds by Theorem \ref{prop:spec_An}. (The points in $\sigma_4$ are indicated by circled dots.) For similar figures for other values of $n$ see \cite{Li:Habil}.}
\label{fig:pic5in12}
\end{figure}

In Figure \ref{fig:30pics2} we plot the sets $\sigma_n$, for $n=5,10,...,30$ (note that each set $\sigma_n$ is invariant under reflection in either axis or under rotation by $90^0$, by Lemma \ref{lem_symm}, and see \cite{CWChonchaiyaLindner2011} for smaller plots of these sets for $n=1,...,30$). By the above theorem, $\sigma_n \subset \pi_{2n+2}$ for each $n$, so that
$$
\sigma_\infty := \bigcup_{n\in\N} \sigma_n \subset \pi_\infty.
$$
The inclusion $\sigma_n\subset \pi_{2n+2}$ is illustrated for $n=4$ in Figure \ref{fig:pic5in12}.

An interesting question, alluded to already in Section \ref{sec:lo}, is whether $\pi_\infty$, which is contained in $\Sigma$, or $\sigma_\infty$, which is a countable subset of $\pi_\infty$, are dense in $\Sigma$, the spectrum of $A^b$ for $b$ pseudo-ergodic. Of course, we do not know what $\Sigma$ is, so that this question is difficult to resolve! We do know however (Theorem \ref{thm_sier}) that the unit disc  $\D\subset \Sigma$, and we can consider the question as to whether $\pi_\infty$ or $\sigma_\infty$ are dense in $\D$.
Recall that the sets $\pi_n$, for $n=5,10,...,30$, are plotted already in Figure \ref{fig:30pics1}. Studying Figures \ref{fig:30pics1} and \ref{fig:30pics2}, it appears that there is a ``hole'' in both $\sigma_n$ and $\pi_n$ around the origin, though these holes appear to be reducing in size as $n$ increases. And in fact, as mentioned already in Section \ref{sec:lo}, it has been shown recently that $\pi_\infty$ is dense in $\D$. Further, it appears to us plausible, comparing the two figures, to conjecture that $\sigma_\infty$ is dense in $\pi_\infty$ and so dense in $\D$.

\noindent
\begin{center}
\ifpics{\includegraphics[width=0.8\textwidth]{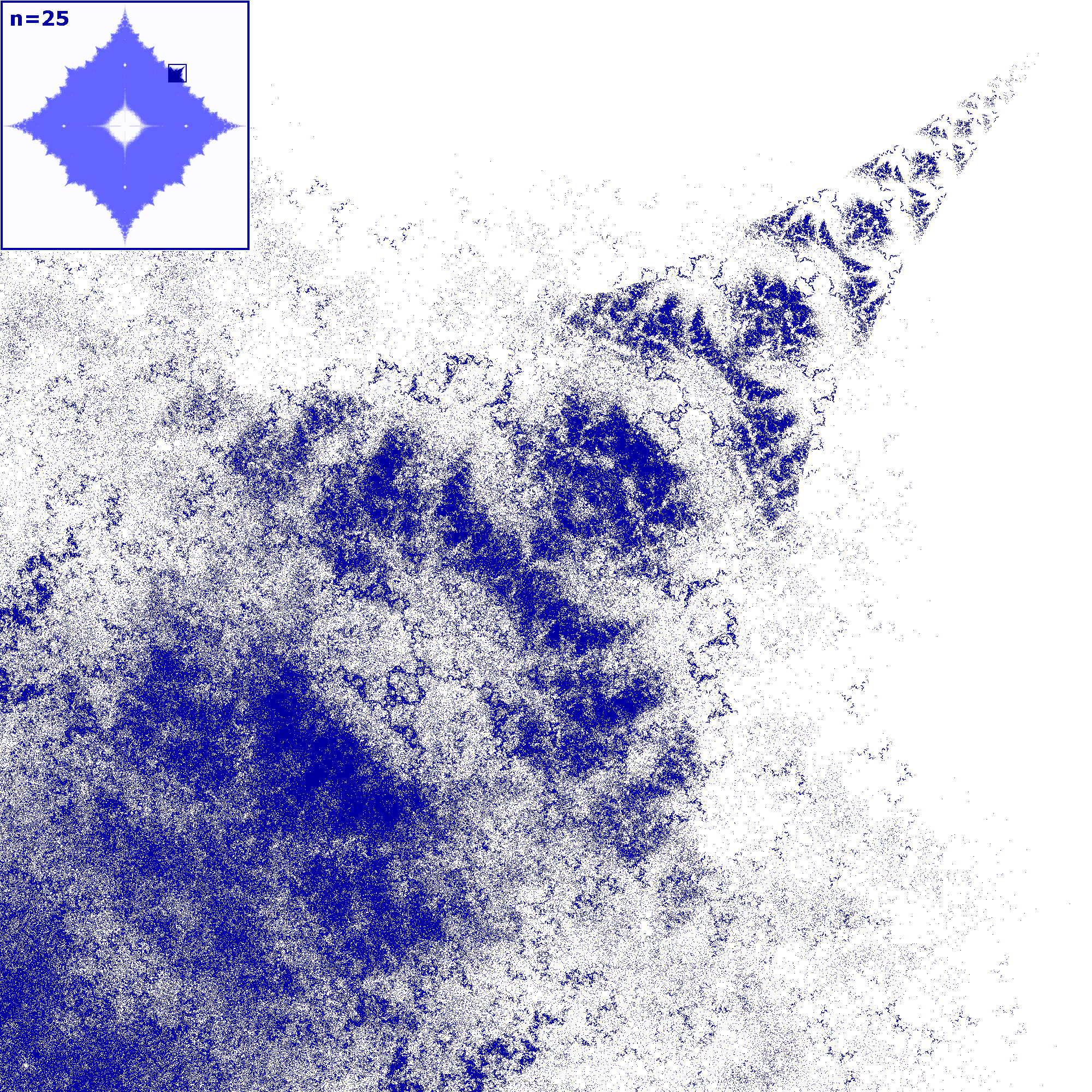}}\fi
\end{center}
\begin{figure}[h]
\caption{\footnotesize This is a zoom into $\sigma_{25}$ -- the 5th
picture of Figure \ref{fig:30pics2}. The location of this zoom is
near the point $1+i$, which is the midpoint of the northeast edge of
the square $W(A^b) = \Delta$. The
picture clearly suggests self-similar features of the set $\sigma_{25}$.}
\label{fig:pic25zoom}
\end{figure}

Figure \ref{fig:pic25zoom}, taken from \cite{CWChonchaiyaLindner2011}, zooms into the part of the set $\sigma_{25}$ around $1+\ri$. Intriguingly this set, the collection of all eigenvalues of a set of $2^{24}$ matrices of size $25\times 25$  ($25\times 2^{24}=419,430,400$ eigenvalues in all!), appears to have a self-similar structure. We have no explanation for these beautiful geometrical patterns, and it is not clear to us how to gain insight into the geometry of this set.

In the next theorem and corollary we show the analogue  of Theorem \ref{prop:spec_An}
for pseudospectra.

\begin{theorem}\label{prop:ineq resolvent}
If $b$ is pseudo-ergodic and $n\in\N$ then, for all $\lambda\in\C\setminus \Sigma$, $f\in \{\pm 1\}^{n-1}$, and $p\in [1,\infty]$,
\[
\|(A^{f}_{n}-\lambda I_{n})^{-1}\|_p\ \leq\ \|(A^{b}-\lambda
I)^{-1}\|_p.
\]
\end{theorem}
\begin{proof}
Let $\lambda\in\C\setminus \Sigma$ and $f\in \{\pm 1\}^{n-1}$, so that $A^{f}_{n}-\lambda I_{n}$ is invertible by Theorem \ref{prop:spec_An}, and let $p\in [1,\infty]$.
Put $M:=\|(A^{f}_{n}-\lambda I_{n})^{-1}\|_p$. For every $\delta>0$,
there exists an $x=(x_1,...,x_n)^\top\in\C^n$ such that
$\left\|x\right\|_p=1$ and $y:= (A^{f}_{n}-\lambda I_{n})x$ has
$\left\|y\right\|_p<{\frac{1}{M-\delta}}$. Now let $c,d\in\{\pm
1\}^{\Z}$ be the sequences in the matrix  $A^{c,d}$ introduced in the
proof of Theorem \ref{prop:spec_An}. At this point our current proof has
to bifurcate depending on the value of $p$.

\noindent\underline{Case 1}: $p=\infty$ \\
Define $\widetilde{x}\in\ell^{\infty}(\Z)$ exactly as in the proof
of Theorem \ref{prop:spec_An}. Then
$\widetilde{y}:=(A^{c,d}-\lambda I)\widetilde{x}$ is of the form
$\widetilde{y}= (\cdots ,y^\top,0,(J_n y)^\top,0,y^\top,0,(J_n
y)^\top, \cdots)^\top\in \ell^{\infty}(\Z)$ and
$\left\|\widetilde{y}\right\|_{\infty}=\left\|y\right\|_{\infty}$,
as well as
$\left\|\widetilde{x}\right\|_{\infty}=\left\|x\right\|_{\infty}$,
so that
\[
\left\|(A^{c,d}-\lambda
I)^{-1}\right\|_\infty\geq{\frac{\left\|\widetilde{x}\right\|_{\infty}}{\left\|\widetilde{y}\right\|_{\infty}}=\frac{\left\|x\right\|_{\infty}}{\left\|y\right\|_{\infty}}>M-\delta}.
\]

\noindent\underline{Case 2}: $p<\infty$\\
For any $m\in\N$, let $\widetilde{x}^{(m)}$ be the sequence
$\widetilde{x}$ from case 1, but with all entries of index outside
$\{-m(n+1),\ldots,m(n+1)\}$ put to zero (where we suppose that the sequence $\widetilde x$ is numbered so that at
index zero there is one of the $0$ entries between $x$ and
$\widehat x$ of $\widetilde x$, so that $\widetilde x^{(m)}_0=0$). Then
$\widetilde{y}^{(m)}:=(A^{c,d}-\lambda I)\widetilde{x}^{(m)}$ is the
same as $\widetilde{y}$ from Case 1 for entries with index between
$-m(n+1)+1$ and $m(n+1)-1$, is zero outside
$\{-m(n+1)-1,\ldots,m(n+1)+1\}$ and we have
$\widetilde{y}^{(m)}_{-m(n+1)}=x_1$ if $m$ is even and $\widetilde{y}^{(m)}_{-m(n+1)}=x_n$ if
$m$ is odd, while $\widetilde{y}^{(m)}_{m(n+1)}=-x_1$ if $m$ is even and
$\widetilde{y}^{(m)}_{m(n+1)}=-x_n$ if $m$ is odd. As a result, we find that
\[
\|\widetilde{x}^{(m)}\|_{p}^{p}=2m\left\|x\right\|_{p}^{p}
\qquad\textrm{and}\qquad
\|\widetilde{y}^{(m)}\|_{p}^{p}=2m\left\|y\right\|_{p}^{p}+\left\{\begin{array}{cl}
2\,|x_1|^{p} &\textrm{if
$m$ is even},\\
2\,|x_n|^{p}&\textrm{if $m$ is odd.}\end{array}\right.
\]
From $\|x\|_{p}=1$, $\|y\|_p<\frac{1}{M-\delta}$ and
$|x_1|,|x_n|\leq\|x\|_{p}=1$ we hence get that
$\|\widetilde{x}^{(m)}\|_p=2m$ and
$\|\widetilde{y}^{(m)}\|_{p}^{p}<2m\frac{1}{(M-\delta)^{p}}+2$, so that
\[
\|(A^{c,d}-\lambda I)^{-1}\|_p^{p}\ \geq\
\frac{\|\widetilde{x}^{(m)}\|_{p}^{p}}{\|\widetilde{y}^{(m)}\|_{p}^{p}}
\ >\ \frac{2m}{\frac{2m}{(M-\delta)^{p}}+2}\ =\
\frac{1}{\frac{1}{(M-\delta)^{p}}+\frac{1}{m}}.
\]

In either case, Case 1 or 2, these inequalities hold for all
$\delta>0$ and all $m\in\N$. Hence,  and applying Lemma \ref{lem_sim},
$$
\|(A_{n}^{f}-\lambda I_{n})^{-1}\|_p = M \leq \|(A^{c,d}-\lambda
I)^{-1}\|_p = \|(A^{cd}-\lambda
I)^{-1}\|_p \leq \|(A^{b}-\lambda
I)^{-1}\|_p,
$$
where the last inequality follows by \cite[Theorem 5.12(ix)]{CWLi2008:Memoir}, since
$A^{cd}$ is a limit operator of $A^b$ by Lemma \ref{pe_lo}.
\end{proof}

The following corollary is immediate.

\begin{corollary}\label{cor:union_sp2}
If $b$ is pseudo-ergodic and $n\in\N$ then, for all $\eps>0$ and
$p\in [1,\infty]$,
\[
\sigma_{n,\eps}^p := \bigcup_{c\in\{\pm 1\}^{n-1}}\specn_\eps^p\, A_{n}^{c}\ \subseteq\
\specn_\eps^p\, A^{b} = \Sigma_\eps^p,
\]
and in particular $\specn_\eps^p\, A_{n}^{b}\ \subset\
\specn_\eps^p\, A^{b}$.
\end{corollary}

\subsection{Convergence of the finite matrix spectral sets to their infinite matrix counterparts} \label{sec:conv}

As we have remarked at the beginning of this section, it is not clear that the spectrum of a general banded matrix should have anything to do with the spectra of its finite submatrices. In particular, it need not be the case either that the spectrum of a large finite submatrix is contained in a neighbourhood of the spectrum of the corresponding infinite matrix, or that the converse statement is true. But the situation is somewhat more positive for the pseudospectrum, namely that, as we show for a general tridiagonal matrix as our first result of this section (and our method of argument applies to banded matrices more generally), the $\eps$-pseudospectrum of the infinite matrix  is contained in the $\eps^\prime$-pseudospectrum of an  appropriately chosen $n\times n$ submatrix, for a given $\eps^\prime>\eps$, provided $n$ is sufficiently large. The argument is based on a standard and rather obvious idea: the point is that every eigenvector, or approximate eigenvector, of the infinite matrix is, when truncated in a careful way, also an approximate eigenvector of the finite matrix.

The opposite statement is, in general, false; an approximate eigenvector of a large finite matrix is an approximate eigenvector also of an infinite matrix $B$, but $B$ 
need not be the infinite matrix whose spectrum one wishes to approximate! (One recent result which expresses this idea very precisely in the $\ell^2$ case for a version of the finite section method for the class of general pseudo-ergodic tridiagonal matrices is \cite[Theorem 2.14]{LindnerRoch2010}.) But, for the pseudo-ergodic operators $A^b$ and $A^b_+$ that we are studying, we have also shown, in Corollary \ref{cor:union_sp2} and Theorem \ref{thm_main}, that $\specn_\eps^p\, A_{n}^{b}\ \subset
\specn_\eps^p A^{b}= \specn_\eps^p A_+^p$. Putting this result together with Theorem \ref{thm_pseudo_converg} proves that $\specn_\eps^p A_n^b \nearrow \specn_\eps^p A^b$ (using the notation of the introduction and Section \ref{sec:pseudo}).

\begin{theorem} \label{thm_pseudo_converg} Suppose that $A=(a_{ij})_{i,j\in \Z}$ is a bi-infinite tridiagonal matrix with $M:= \sup_{ij}|a_{ij}|<\infty$. Define the
semi-infinite matrix $A_+$ by $A_+=(a_{ij})_{i,j\in\N}$ and, for $\ell,m\in \N$ with $\ell\leq m$, define the finite matrix of order $m+1-\ell$ by $A_{\ell,m} = (a_{ij})_{i,j\in \{\ell,...,m\}}$. Then, for every $\eps^\prime>\eps>0$ and $p\in [1,\infty]$, there exists $N\in \N$ such that
\begin{equation} \label{eq:inc1}
\specn_\eps^p A \subset \specn_{\eps^\prime}^p A_{\ell,m}, \quad \mbox{ for } \ell \leq -N \mbox{ and } m \geq N,
\end{equation}
and
\begin{equation} \label{eq:inc2}
\specn_\eps^p A_+ \subset \specn_{\eps^\prime}^p A_{1,m}, \quad \mbox{ for } m \geq N.
\end{equation}
\end{theorem}
\begin{proof} We will prove (\ref{eq:inc1}). The proof of (\ref{eq:inc2}) is similar.


As a first step we will show that, given some $\eps>0$ and $p\in [1,\infty]$, for  every  $\lambda \in \specn_\eps^p A$ there exists $N\in\N$ (depending on $\lambda$) such that $\lambda \in \specn_\eps^p A_{\ell,m}$ if $\ell\leq -N$ and $m\geq N$. We then combine this result with a compactness argument to obtain the proof of the theorem.

So suppose that $\eps>0$, $p\in [1,\infty]$, and that $\lambda \in \specn_\eps^p A$. Then, by Theorem \ref{thm_pse}(iv), either $\lambda\in \sppten^p A$ or $\bar \lambda \in \sppten^q A^*$, where $p^{-1}+q^{-1}=1$.

Suppose first that $\lambda\in \sppten^p A$, i.e. that there exists $x\in \ell^p(\Z)$ with $\|x\|_p=1$ and $\tilde \eps :=\|y\|_p < \eps$, where $y:=(A-\lambda I)x$. In the case $p<\infty$, let $\widetilde x := (x_\ell,...,x_m)^T$ and $\widetilde y := (A_{\ell,m}-\lambda I_{m+1-\ell})\widetilde x$, so that $\widetilde y_k = y_k$, $k= \ell+1,...,m-1$. Since $|x_k|\to 0$ as $|k|\to\infty$, it is easy to see that we can, given $\delta>0$, choose $N$ such that $\|\widetilde y\|_p<\tilde \eps+\delta$ and $\|\widetilde x\|_p > 1 - \delta$ whenever $\ell \leq -N$ and $m\geq N$. But this implies that $\lambda\in \specn_\eps^p A_{\ell,m}$ if $N$ is large enough and  $\ell \leq -N$ and $m\geq N$. In the case $p=\infty$ we have to modify this argument slightly. Given $\ell \leq -N$ and $m \geq N$ put $\widetilde x = (\widetilde x_\ell, ...\widetilde x_m)^T$ with $\widetilde x_k := \omega_N(k) x_k$, $k=\ell,...,m$, and $\omega_N(k) := \max(0,1-|k|/(N-1))$, $k\in\Z$, and let $\widetilde y := (A_{\ell,m}-\lambda I_{m+1-\ell})\widetilde x$. Then, for $i= \ell,...,m$,
\begin{eqnarray*}
|\widetilde y_i| &= &\big|a_{i,j-1} \widetilde x_{j-1} + a_{ij} \widetilde x_j + a_{i,j+1} \widetilde x_{j+1}\big|\\
 &= &\big|\omega_N(j) y_j + a_{i,j-1} (\omega_N(j-1)-\omega_N(j))x_{j-1} + a_{i,j+1} (\omega_N(j+1)-\omega_N(j)) x_{j+1}\big|\\
 & \leq & |y_i| + 2M \|x\|_\infty/(N-1) < \eps + 2M/(N-1),
\end{eqnarray*}
since $|\omega_N(j)-\omega_N(j+1)| \leq (N-1)^{-1}$, for $k\in \Z$. Since also, for each $k\in \ell,...,m$, $\widetilde x_k \to x_k$ as $N\to \infty$, it is clear that, for every $\delta>0$, if $N$ is chosen large enough, then $\|\widetilde x\|_\infty \geq 1-\delta$, and also $\|\widetilde y\|_\infty < \tilde \eps + \delta$. But this implies that $\lambda\in \specn_\eps^p A_{\ell,m}$ if $N$ is large enough and  $\ell \leq -N$ and $m\geq N$.

If $\bar \lambda\in \sppten^q A^*$ then essentially the identical argument shows that $\bar \lambda\in \specn_\eps^q A^*_{\ell,m}$. But this implies that $\lambda \in \specn_\eps^p A_{\ell,m}$ \cite[Section 4]{TrefEmbBook}. This completes the proof of the first step.

To finish the proof of the theorem we argue as follows. Given $\eps^\prime > \eps>0$ and $p\in [1,\infty]$, let $\eta := (\eps^\prime - \eps)/2$, and $\eps^* = \eps+\eta$.  Let $S := \overline{\specn_\eps^p A}$, and let $O := \{\lambda+\eta \D: \lambda \in S\}$. Then $O$ is an open cover of the compact set $S$, and so has a finite subcover, i.e. there exists a finite set $\Lambda\subset \overline{\specn_\eps^p A}$ with $S\subset \bigcup_{\lambda\in \Lambda} (\lambda+\eta\D)= \eta\D + \Lambda$.
Now $\Lambda \subset \overline{\specn_\eps^p A} \subset \specn_{\eps^{*}}^p A$. Applying the result shown in the first step, we see that we can choose $N$ so that, for $\ell\leq -N$ and $m\geq N$, $\Lambda \subset \specn_{\eps^*}^p A_{\ell,m}$. Thus $\specn_\eps^p A \subset S\subset \eta\D+\Lambda \subset \eta\D + \specn_{\eps^*}^p A_{\ell,m}\subset \specn_{\eps^\prime}^p A_{\ell,m}$, by (\ref{eq_int4}).
\end{proof}

To apply this result, for $\ell,m\in\Z$ with $\ell\leq m$, let $A^b_{\ell,m}$ denote $A_{\ell,m}$, the matrix of order $m+1-\ell$ as defined in the above theorem, in the case that $A=A^b$. So, in particular, $A^b_{1,n} = A^b_n$ for $n\in \N$.

\begin{corollary} \label{cor_pse_conv}  If $b\in \{\pm 1\}^\N$ is pseudo-ergodic then, for every $\eps>0$ and $p\in [1,\infty]$,
$$
\specn_\eps^p A^b_n = \specn_\eps^p A^b_{1,n} \nearrow \specn_\eps^p A_+^b = \Sigma_\eps^p, \quad \mbox{ as } n\to\infty.
$$
If $b\in \{\pm 1\}^\Z$ is pseudo-ergodic then, for every $\eps>0$ and $p\in [1,\infty]$,
$$
\specn_\eps^p A^b_{\ell,m} \nearrow \specn_\eps^p A^b = \Sigma_\eps^p, \quad \mbox{ as } \ell\to-\infty \mbox{ and } m\to\infty.
$$
\end{corollary}
\begin{proof}
We will prove the second of these statements. The proof of the first is similar.
From Corollary \ref{cor:union_sp2} and Theorem \ref{thm_pseudo_converg}, given any $\eps^\prime\in (0,\eps)$ there exists $N\in \N$ such that
\begin{equation} \label{eq:inc8}
\specn_{\eps^\prime}^p A \subset \specn_{\eps}^p A_{\ell,m} \subset \specn_{\eps}^p A, \quad \mbox{ for } \ell \leq -N \mbox{ and } m \geq N.
\end{equation}
Since, from (\ref{eq_pseconv}), $\specn_{\eps^\prime}^p A \nearrow \specn_{\eps}^p A$ as $\eps^\prime\to \eps^-$, it follows that  $\specn_\eps^p A^b_{\ell,m} \nearrow \specn_\eps^p A^b$ (which is equal to $\Sigma_\eps^p$ by Theorem \ref{thm_main}), as $\ell\to-\infty$ and $m\to\infty$.
\end{proof}

A similar result holds for the convergence of the numerical range, as an instance of the general result Theorem \ref{thm:nrconv}. Note that while convergence of the pseudospectra needs that $b$ is pseudo-ergodic, to ensure that the matrix pseudospectra are contained in the operator pseudospectra, the corresponding inclusion (\ref{eq:numr}) for numerical ranges holds for {\em any} bounded linear operator, so that we need no constraint on $b$. The following is thus an immediate corollary of Theorem \ref{thm:nrconv} and Lemma \ref{lem_nr}.
\begin{corollary} \label{cor_numconv} If $b\in \{\pm 1\}^\N$ then
$$
W(A^b_n) = W(A^b_{1,n}) \nearrow W(A^b_+), \quad \mbox{ as } n\to\infty,
$$
with $W(A^b_+)=\Delta$ if $b$ is pseudo-ergodic.
If $b\in \{\pm 1\}^\Z$  then
$$
W(A^b_{\ell,m}) \nearrow W(A^b), \quad \mbox{ as } \ell\to-\infty \mbox{ and } m\to\infty.
$$
with $W(A^b)=\Delta$ if $b$ is pseudo-ergodic.
\end{corollary}

\subsection{Quantitative convergent approximations to the spectrum and pseudospectrum} \label{sec:final}

In this section we present numerical algorithms for approximating $\Sigma$ and $\Sigma_\eps^2$ which are, respectively, from Theorem \ref{thm_main}, the spectrum and the $\ell^2$ $\eps$-pseusdospectrum of both $A^b$ and $A^b_+$ in the case when $b$ is pseudo-ergodic.

The previous subsection already provides potential methods for computing these sets. We have that, if $b\in\{\pm1\}^n$ is pseudo-ergodic, then
\begin{equation} \label{eq_pconv}
\Sigma_\eps^2 = \lim_{n\to\infty} \specn_\eps^2 A^b_n.
\end{equation}
This then implies, by (\ref{eq_pseconv}), that
\begin{equation} \label{eq_sconv}
\Sigma = \lim_{\eps\to 0} \, \lim_{n\to\infty} \specn_\eps^2 A^b_n.
\end{equation}
In principle, these equations can be used as the basis of algorithms for computing $\Sigma_\eps^p$ and $\Sigma$. In particular, to approximate $\Sigma_\eps^2$ one uses the sequence of sets $\specn_\eps^2 A^b_n$, $n=1,2,...$, which can be computed as described in Section \ref{sec:pseudo}. The difficulty with this scheme is that one has no idea of the rate of convergence of $\specn_\eps^2 A^b_n$ to $\Sigma$. Indeed it is clear that it can be arbitrarily slow: to see this consider that if $c\in \{\pm 1\}^\N$ is pseudo-ergodic, then so is $b\in \{\pm 1\}^\N$ if $b_m=c_m$ for all sufficiently large $m$. But this means that it can hold that $b$ is pseudo-ergodic and that $b_m=1$, for $1\leq m\leq N$, with $N$ arbitrarily large. If this is the case then  $A^b_n$ is self-adjoint and thus, and by Lemma \ref{lem_nr}, $\spec A^b_n \subset (-2,2)$ and $\specn_\eps^2 A^b_n = \spec A^b_n + \eps \D \subset (-2,2)+\eps \D$, for $n\leq N$. So if, e.g., $N=10^9$ then, while ultimately $\specn_\eps^2 A^b_n \to \Sigma_\eps^2$, there is no early sign of this.

The situation with (\ref{eq_sconv}) is rather worse. This equation implies that there exists some sequence of positive reals $\eps_n\to 0$ for which it holds that
$$
\specn_{\eps_n}^2 A^b_n \to \Sigma,
$$
but provides neither a recipe for choosing the $\eps_n$ nor any guarantee of the rate of convergence.

The source of the difficulty regarding the rate of convergence can be traced back to Theorem \ref{thm_pseudo_converg} and its proof, this theorem  a key ingredient in the proof of Corollary \ref{cor_pse_conv} and so of (\ref{eq_pconv}). This theorem guarantees that, for every $\eps^\prime>\eps>0$, $\Sigma_\eps^2 = \specn_\eps^2 A^b_+ \subset \specn_{\eps^\prime}^2 A^b_n$ for all $n$ sufficiently large, but gives no idea of how large $n$ should be. And indeed we have argued above that there is no upper bound on how large $n$ may need to be for this equation to hold for a given pseudo-ergodic $b$.

This difficulty has been resolved in recent work by the authors \cite{CW.Heng.ML:UpperBounds}, who quantify, for general tridiagonal matrices, by a sharpened version of the arguments of Theorem \ref{thm_pseudo_converg}, adapted particularly to the case $p=2$, exactly how $\eps^\prime$ should depend on $n$ in (\ref{eq:inc2}), but at the expense of replacing in this equation the pseudospectrum of a single $n\times n$ submatrix by the union of the pseudospectra of {\em all} possible $n\times n$ principal submatrices. The results in \cite{CW.Heng.ML:UpperBounds} are much more general, but we will restrict the exposition here to how these results apply to the bi-infinite matrix $A^b$ with $b\in \{\pm 1\}^Z$. Using the notation of Corollary \ref{cor_pse_conv}, the result shown in \cite{CW.Heng.ML:UpperBounds} (or see \cite[Corollary 3.7]{HengPhD}) is the following when applied to $A^b$:

\begin{theorem} \label{thm_upper} For $b\in \{\pm 1\}^\Z$, $\eps>0$, and $n\in\N$,
$$
\specn_\eps^2 A^b \subset \bigcup_{\ell\in\Z} \specn_{\eps+\eps_n}^2 A_{\ell,\ell+n-1}^b,
$$
where $\eps_{n} = 4\sin \theta_n \leq 2\pi/(n+2)$, with $\theta_n$ the unique solution in the interval $\displaystyle{\left(\frac{\pi}{2(n+3)},\frac{\pi}{2(n+2)}\right]}$ of the equation
$
2\cos\left((n+1)\theta\right)\ =\ \cos\left((n-1)\theta\right)$. Further,
$$
\spec A^b \subset \overline{\bigcup_{\ell\in\Z} \specn_{\eps_n}^2 A_{\ell,\ell+n-1}^b} = \bigcup_{\ell\in\Z} {\mathrm{Spec}}_{\eps_n}^2 A_{\ell,\ell+n-1}^b.
$$
\end{theorem}

An important point is that the unions of pseudospectra over $\ell\in\Z$ in the above equations reduce to finite unions, because there are only $2^{n-1}$ distinct $n\times n$ matrices $A^c_n$ with $c\in \{\pm 1\}^{n-1}$. In the notation introduced in Corollary \ref{cor:union_sp2}, it must hold that
$\bigcup_{\ell\in\Z} \specn_{\eta}^2 A_{\ell,\ell+n-1}^b \subset \sigma_{n,\eta}^2$, for every $\eta>0$. For small values of $n$, $\eps_n$ in the above theorem can be calculated explicitly, in particular
\begin{equation} \label{eq_smalleps}
\eps_1 = 2 \mbox{ and } \eps_2 = \sqrt{2}.
\end{equation}

\begin{example} As a first example of application of the above theorem, consider the case when $b_m=1$ for each $m$. Then $A_{\ell,\ell+n-1}^b = A_{1,n}^b = A^b_n$ for each $\ell$. Further, this matrix is self-adjoint, so that $\specn_\eta^2 A^b_n = \spec A_n^b + \eta \D$, for every $\eta>0$. Thus the statements of the theorem reduce to
\begin{equation} \label{eq:ex}
\spec A^b \subset \spec A_n^b + \eps_n\overline{\D} \; \mbox{ and } \; \specn_\eps^2 A^b \subset \spec A_n^b + (\eps+\eps_n)\D, \; \eps>0.
\end{equation}
In this simple case we can compute the above sets explicitly, to check that the above inclusions hold, finding that $\spec A^b = [-2,2]$, $\specn_\eps^2 A^b = [-2,2]+\eps\D$, and $\spec A_n^b = \left\{2\cos\frac{j\pi}{n+1}:j=1,...,n\right\}$. Elementary calculations show that the inclusions (\ref{eq:ex}) do hold in this case, in fact one can calculate (see \cite[Section 3.2.2]{HengPhD} for details), if $\eps_n$ were replaced with $\eps_n^*\leq\eps_n$ in the above inclusions, the smallest value of $\eps_n^*$ for which the inclusions would still hold. This is $\eps_1^*=2$, $\eps_n^* = 2\sin(\pi/(2(n+1)))$ if $n$ is even (in particular $\eps_2^* = 1$), and $\eps_n^* = \sin(\pi/(n+1))$ if $n\geq 3$ is odd. Thus $\eps_n/\eps_n^*=1$ for $n=1$ (the bound (\ref{eq:ex}) is sharp for $n=1$) and $\eps_n/\eps_n^*\to 2$ as $n\to\infty$.
\end{example}

The main example of interest to us here is the case where $b\in \{\pm 1\}^\Z$ is pseudo-ergodic. Recall from Theorem \ref{thm_main} that $\spec A^b = \Sigma$ and $\specn_\eps^2 A^b = \Sigma_\eps^2$ in that case. Combining Theorem \ref{thm_upper} with Theorem \ref{prop:spec_An}, Corollary \ref{cor:union_sp2}, (\ref{eq:thing}), (\ref{eq_pseconv}) and (\ref{eq_pseconv2}), we obtain the following result.

\begin{theorem} \label{thm_final} For $\eps>0$ and $n\in\N$,
\begin{equation} \label{eq:incf1}
\sigma_n \subset \Sigma \subset \overline{\sigma_{n,\eps_n}^2}  \subset \overline{\Sigma_{\eps_n}^2} \mbox{ and } \Sigma_\eps^2 \subset \sigma_{n,\eps+\eps_n}^2 \subset \Sigma_{\eps+\eps_n}^2,
\end{equation}
where $\eps_n$ is defined as in Theorem \ref{thm_upper}. Further,
\begin{equation} \label{eq:incf2}
\sigma_n + \eps \D \subset \Sigma_\eps^2, \mbox{ for } \eps>0, \;\mbox{ and }\; \Sigma_{\eps-\eps_n}^2 \subset \sigma_{n,\eps}^2 \subset \Sigma_{\eps}^2, \mbox{ for } \eps>\eps_n.
\end{equation}
Moreover, as $n\to\infty$,  $\overline{\sigma_{n,\eps_n}^2}\searrow \Sigma$ and, for $\eps>0$, $\sigma_{n,\eps+\eps_n}^2\searrow \Sigma^2_\eps$ and $\sigma_{n,\eps}^2\nearrow \Sigma^2_\eps$.
\end{theorem}

In most respects this result is superior to Corollary \ref{cor_pse_conv}. It provides both upper and lower bounds for $\Sigma_\eps^2$, moreover these converge to $\Sigma_\eps^2$ as $n\to\infty$ at guaranteed convergence rates (at least as fast as $\Sigma_{\eps-\eps_n}$ and $\Sigma_{\eps+\eps_n}$, respectively). Further, the theorem provides an upper bound which is convergent to $\Sigma$, at least as fast as $\Sigma_{\eps_n}^2$. Of course, that the convergence rates are guaranteed is at a cost: evaluating $\sigma^2_{n,\eta}$ for some $n\in\N$ and $\eta>0$ requires exponentially large computation for $n$ large. Precisely, using the characterisation (\ref{eq:sing}), we see that
\begin{equation} \label{eq:sn}
\sigma^2_{n,\eta} = \{\lambda \in \C: S_n(\lambda) < \eta\},
\end{equation}
where
\begin{equation} \label{eq:Sndef}
S_n(\lambda) := \min_{c\in \{\pm 1\}^{n-1}} s_{\mathrm{min}}(A^c_n-\lambda I_n), \quad \lambda \in \C.
\end{equation}
Clearly, computing $S_n(\lambda)$ for a particular $\lambda$, to check membership of $\sigma^2_{n,\eta}$, requires calculation of the smallest singular value of $2^{n-1}$ matrices of order $n$. Note that it follows from (\ref{eq:contdep}) that
\begin{equation} \label{eq:contdep2}
|S_n(\lambda)-S_n(\mu)| \leq |\lambda-\mu|, \quad \lambda,\mu\in\C.
\end{equation}

\noindent
\ifpics{  
\begin{center}
\begin{tabular}{ccc}
\includegraphics[width=0.3\textwidth]{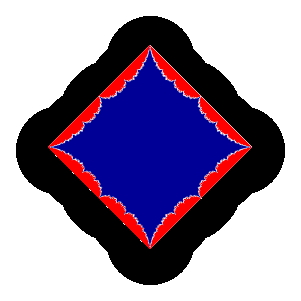}&
\includegraphics[width=0.3\textwidth]{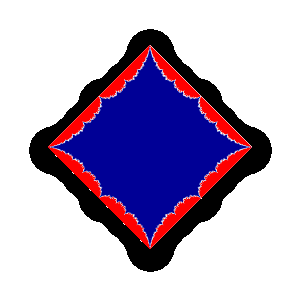}&
\includegraphics[width=0.3\textwidth]{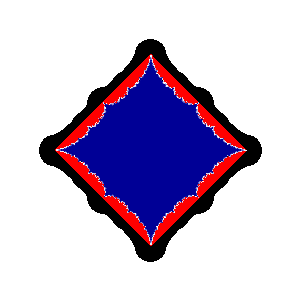}
\end{tabular}
\end{center}
}\fi 
\begin{figure}[h]
\caption{\footnotesize \label{sigma_method1} Plots, for $n=6,12$ and $18$, of the sets $\overline{\sigma^2_{n,\eps_n}}$, which are inclusion sets for $\Sigma=\spec A^{b}$,
 when $b\in\{\pm 1\}^{\Z}$ is pseudo-ergodic. Also shown, overlaid in red, is the square $\Delta$, with corners at $\pm 2$ and $\pm 2\ri$, which is $W(A^b)$, the numerical range of $A^b$. Overlaid on top of that in blue is the set $\pi_{30}\cup\D$ which, by definition and Theorem \ref{thm_sier}, is a subset of $\Sigma$.} \label{fig:inclusion}
\end{figure}

In Figure \ref{sigma_method1} we plot $\overline{\sigma^2_{n,\eps_n}}$, for $n=6,12$, and $18$. Each of these sets contains $\Sigma$, by Theorem \ref{thm_final}, and note that each set is invariant under reflection in either axis or under rotation by $90^0$, by Lemma \ref{lem_symm}. On the same figure we plot the square $\overline{\Delta}$ which, by Lemma \ref{lem_nr}, also contains $\Sigma$. It appears that, for $n\leq 18$, $\Delta\subset \sigma^2_{n,\eps_n}$. If this were to hold for all $n\in\N$ then it would follow, from Theorem \ref{thm_final}, which tells us that $\overline{\sigma^2_{n,\eps_n}}\searrow \Sigma$, and Lemma \ref{lem_nr}, which tells us that $\Sigma\subset \overline{\Delta}$, that $\Sigma=\overline{\Delta}$. It seems impossible from these plots to take an educated guess as to whether or not $\Delta\subset \sigma^2_{n,\eps_n}$ holds for all $n$, not least because the convergence rate of $\overline{\sigma^2_{n,\eps_n}}$ to $\Sigma$ may be slow: Theorem \ref{thm_final} tells us that $\dist(\overline{\sigma^2_{n,\eps_n}},\Sigma)\leq \dist(\Sigma^2_{\eps_n},\Sigma)$ but it follows from (\ref{eq:thing}) that $\dist(\Sigma^2_{\eps_n},\Sigma)\geq \eps_n\approx 2\pi/(n+2)$.

We have not been able to produce similar plots to those in Figure \ref{sigma_method1} for much larger values of $n$ because of the large computational cost. But it is feasible to compute $S_n(\lambda)$ for a single $\lambda$ for larger $n$. We have carried out this computation for $\lambda = 1.5+0.5\ri$, a quarter of the way along one of the sides of $\Delta$. Computing in standard double-precision floating point arithmetic we find that
\begin{equation} \label{eqS34}
S_{34}(1.5+0.5\ri) = 0.17201954132506... > \eps_{34} = 0.169830415547956...\; .
\end{equation}
This implies that $1.5+0.5\ri \not\in \sigma^2_{34,\eps_{34}}$ and so $1.5+0.5\ri \not\in \Sigma$, which of course implies that $\Sigma$ is a strict subset of $\overline{\Delta}$. In fact, in view of (\ref{eq:contdep2}) and the symmetries of $\Sigma$ noted in Lemma \ref{lem_sim}, the inequality (\ref{eqS34}) implies more, namely that
$$
(\pm(1.5\pm0.5\ri) + \eta\D)\cap \Sigma = \emptyset,
$$
for $\eta= \eps_{34}-S_{34}(1.5+0.5\ri) = 0.0021891257771...$.

We note that the computation required to evaluate $S_{34}(1.5+0.5\ri)$ and so establish that $1.5+0.5\ri\not\in \Sigma$ is considerable: we need to evaluate the smallest singular value of $2^{33}\approx 8.6\times 10^9$ matrices of order 34 (of course these computations are ideally suited for parallel implementation). We note that it seems to be necessary to use $n$ as large as 34, in that other computations show that $S_{33}(1.5+0.5\ri) < \eps_{33}$, so that $1.5+0.5\ri \in \sigma^2_{33,\eps_{33}}$.

\section{The Random Case and Concluding Remarks}

We finish this paper by spelling out the implications of the above results for the finite matrices $A^b_n$ and $A^{b,c}_n$, the bi-infinite matrices $A^b$ and $A^{b,c}$, and the semi-infinite matrices $A^b_+$ and $A^{b,c}_+$, in the case when the entries of $b$ and $c$ are random variables taking the values $\pm 1$.

\begin{theorem} \label{thm_rand}
Suppose that the entries of $b\in \{\pm 1\}^\Z$ are iid random variables, with $\mathrm{Pr}(b_m=1)\in (0,1)$. Then:
\begin{description}
\item[(i) ] $\spec A^b \subset \Sigma$, $\spec A^b_+ \subset \Sigma$, with $\spess A^b = \spec A^b = \spess A^b_+=\spec A^b_+ = \Sigma$ almost surely.
\item[(ii)] $W(A^b_+)\subset W(A^b) \subset \Delta$, with  $W(A^b)=W(A^b_+)=\Delta$ almost surely.
\item[(iii)] For $n\in \N$, $\spec A^b_n \subset \Sigma$ and $W(A^b_n)\subset \Delta$, and, as $n\to\infty$, $W(A^b_n)\nearrow \Delta$, almost surely.
\item[(iv)] For $\eps>0$ and $p\in [1,\infty]$, $\specn_\eps^p A^b \subset \Sigma_\eps^p$, $\specn_\eps^p A^b_+ \subset \Sigma_\eps^p$, with $\specn_\eps^p A^b =\specn_\eps^p A^b_+ = \Sigma_\eps^p$ almost surely.
\item[(v)] For $\eps>0$, $p\in [1,\infty]$, and $n\in\N$, $\specn_\eps^p A_n^b \subset \Sigma_\eps^p$ and, as $n\to\infty$, $\specn_\eps^p A^b_n \nearrow \Sigma_\eps^p$, almost surely.
\end{description}
Similarly, if $b,c\in \{\pm 1\}^\Z$, and the entries of $bc$ are iid random variables, with $\mathrm{Pr}(b_mc_m = 1)\in (0,1)$, then (i)-(v) hold with $A^b$, $A^b_+$, $A^b_n$, replaced by $A^{b,c}$, $A^{b,c}_+$, and $A^{b,c}_n$, respectively.
\end{theorem}
\begin{proof}
To see that (i)-(v) hold, note that, by Lemma \ref{imt} and the remarks at the end of Section \ref{sec:lo}, the condition of the theorem imply that $b$ and also $b_+:=(b_1,b_2,...)$ are pseudo-ergodic with probability one. Then (i) follows from the definition of $\Sigma$ in Theorem \ref{thm_lo2}, and from Lemma \ref{lem_ref} and Theorem \ref{thm_main}. That (ii) and (iii) hold follows from Lemma \ref{lem_nr}, Theorem \ref{prop:spec_An} and Corollary \ref{cor_numconv}. That (iv) holds follows from the definition of $\Sigma_\eps^p$ in Theorem \ref{thm_lo2}, and from Lemma \ref{lem_ref} and Theorem \ref{thm_main}. Finally, (v) follows from corollaries \ref{cor:union_sp2} and \ref{cor_pse_conv}. That (i)-(v) hold for the case where $A^b$, $A^b_+$, $A^b_n$ are replaced by $A^{b,c}$, $A^{b,c}_+$, and $A^{b,c}_n$, respectively, and the entries of $bc$ are iid random variables, with $\mathrm{Pr}(b_mc_m = 1)\in (0,1)$, follow using the same results, with the help of lemmas \ref{lem_sim} and \ref{lem_simfd} and the observations between these lemmas on the semi-infinite case, and noting first that these assumptions imply that $bc$ and $(b_1c_1,b_2c_2,...)$ are pseudo-ergodic, with probability one.
\end{proof}

Of course, in the above theorem $\Delta = \{z=a+\ri b: a,b\in\R,\, |a|+|b|<2\}$ and $\Sigma$ and $\Sigma_\eps^p$ are as defined in Theorem \ref{thm_lo2}. The following theorem summarises, for the convenience of the reader, what we have established in the sections above about the compact set $\Sigma$ and the bounded open sets $\Sigma_\eps^p$. Recall that $\sigma_\infty$ and $\pi_\infty$ are defined by (\ref{eqn:sig_sub}) and (\ref{eq:pidef}), respectively, $\sigma^p_{n,\eps}$ is defined in Corollary \ref{cor:union_sp2} (and see (\ref{eq:sn}) for $p=2$), $\eps_n$ is defined in Theorem \ref{thm_upper}, and $S_n(\lambda)$ in (\ref{eq:Sndef}).

\begin{theorem} \label{thm_vfinal}
For $\eps>0$ and $p\in[1,\infty]$, where $q\in[1,\infty]$ is given by $p^{-1}+q^{-1}=1$:
\begin{description}
\item[(i) ] $\overline{\D}\subset \Sigma \subset \overline{\Delta}$, and $\Sigma$ is a strict subset of $\overline{\Delta}$ provided $\eps_{34}-S_{34}(1.5+0.5\ri)>0$, for which see (vi).
\item[(ii)] $\sigma_\infty\subset\pi_\infty\subset \Sigma\subset \Sigma_\eps^p \subset \Sigma^p_{\eps^\prime}$, for $\eps^\prime>\eps$. (See Figures \ref{fig:30pics1}, \ref{fig:30pics2}, \ref{fig:pic5in12} and \ref{fig:pic25zoom}  for visualisations of $\sigma_n$ and $\pi_n$, for $n\in\N$, and their interrelation.)
\item[(iii)] $\Sigma$ and $\Sigma_\eps^p$ are invariant under reflection in the real and imaginary axes and under rotation by $90^0$.
\item[(iv)] $\Sigma_\eps^p = \Sigma_\eps^q$ and $\Sigma_\eps^p\subset \Sigma_\eps^r$ if $1\leq r\leq p\leq 2$, so that $\Sigma_\eps^2 = \bigcap_{r\in[1,\infty]} \Sigma_\eps^r$.
\item[(v)] As $n\to\infty$, $\overline{\sigma_{n,\eps_n}^2}\searrow \Sigma$, $\sigma^2_{n,\eps+\eps_n}\searrow \Sigma_\eps^2$, and $\sigma^2_{n,\eps}\nearrow \Sigma_\eps^2$. (See Figure \ref{fig:inclusion} for visualisations of $\sigma^2_{n,\eps_n}$, for $n=6,12,18$.)
\item[(vi)] For $\lambda = 1.5+0.5\ri$, provided $\eta = \eps_{34}-S_{34}(\lambda)>0$ (and floating point calculations give $\eta\approx 0.00219$), it holds that $\pm(1.5\pm 0.5\ri) + \eta \D \cap \Sigma = \emptyset$.
\end{description}
\end{theorem}
\begin{proof}
Part (i) follows from Theorem \ref{thm_sier} (taken from \cite{CWChonchaiyaLindner2011}) and Lemma \ref{lem_nr}, and that $\Sigma$ is a strict subset of $\Delta$ holds, as discussed at the end of \ref{sec:final}, provided $\eta = \eps_{34}-S_{34}(1.5+0.5\ri)>0$. Part (ii) is Theorem \ref{prop:spec_An}, with  $\Sigma_\eps^p \subset \Sigma^p_{\eps^\prime}$ because $\Sigma_\eps^p = \specn_\eps^p A^b$ if $b\in\{\pm 1\}^\Z$ is pseudo-ergodic (Theorem \ref{thm_main}). Part (iii) is Lemma \ref{lem_symm}, (iv) is from \ref{lem_sim}, (v) is part of Theorem \ref{thm_final}, and (vi) is from the end of Section \ref{sec:final}.
\end{proof}

It is clear from the above results that we understand well, in Theorem \ref{thm_rand}, the interrelation between the numerical ranges and pseudospectra of the semi-infinite, bi-infinite, and finite random matrix cases, and have shown that the almost sure spectrum is the same set $\Sigma$ for the semi-infinite and bi-infinite cases, and contains the spectrum in the finite matrix case. Interesting open questions are whether or not, similarly to the analogous results for the pseudospectra, $\spec A^b_n \nearrow \Sigma$ almost surely as $n\to\infty$, which would imply that $\sigma_\infty$ is dense in $\Sigma$, so that $\pi_\infty$ is dense in $\Sigma$. (That $\sigma_\infty$ is dense in $\Sigma$ was conjectured in \cite{CWChonchaiyaLindner2011}.)  Note that, if it does hold that $\spec A^b_n \nearrow \Sigma$ almost surely, then both Figs \ref{fig:30pics1} and \ref{fig:30pics2} are visualisations of sequences of sets converging to $\Sigma$.

Regarding the geometry of $\Sigma$ (and of the pseudospectra $\Sigma_\eps^p$), we have some information in Theorem \ref{thm_vfinal}, including in the last part of this theorem establishing  a computable sequence of sets converging from above to $\Sigma$ (a sequence of three of these plotted in Figure \ref{fig:inclusion}). However there is much that is not known. Is $\Sigma$ connected (which would imply, by general results on pseudospectra \cite[Theorem 4.3]{TrefEmbBook}, that also $\Sigma_\eps^p$ is connected)? In fact, is $\Sigma$ simply-connected? What is the geometry of the boundary of $\Sigma$, and the geometry of the sets $\sigma_n$, the finite-dimensional analogues of $\Sigma$ (cf.\ Figure \ref{fig:pic25zoom})? We have conjectured in \cite{CWChonchaiyaLindner2011} that $\Sigma$ is a simply-connected set which is the closure of its interior and which has a fractal boundary, which is plausible from, or at least consistent with, Figure \ref{fig:inclusion}, if it holds that $\overline{\sigma_\infty} = \Sigma$. Our methods and results provide no information about what is a usual concern of research on random matrices, to obtain asymptotically in the limit as $n\to\infty$ the pdf of the density of eigenvalues, except, of course, that we have shown in Theorem \ref{thm_rand}(iii) that the support of this pdf is a subset of $\Sigma$.

There are many possibilities for applying the methods introduced in this paper to much larger classes of random (or pseudo-ergodic) operators. For some steps in this direction we refer the reader to \cite{LindnerRoch2010,CW.Heng.ML:UpperBounds,CWDavies2011}.


{\bf Acknowledgements.} We are grateful to Estelle Basor from the American Institute of Mathematics for drawing our attention to this beautiful operator class, and are grateful for feedback on our work in progress from Brian Davies and Eugene Shargorodsky (KCL) and from Titus Hilberdink and Michael Levitin (Reading), including the feedback that prompted the computation of $S_{34}(1.5+0.5\ri)$ in Section \ref{sec:final}. In regard to this computation, we are grateful for the assistance with programming and parallel implementation of the calculations from Roman Unger of TU Chemnitz. We wish to thank Albrecht B\"ottcher (Chemnitz) for his very helpful comments on an earlier version of this script. We also acknowledge the financial support of a Higher Education Strategic Scholarship for Frontier Research from the Thai Ministry of Higher Education to the second author, of Marie-Curie Grant MEIF-CT-2005-009758 of the EU to the third and first authors, and Marie-Curie Grant PERG02-GA-2007-224761 of the EU to the third author.


\bigskip

\noindent {\bf Authors' addresses:}\\
\\
Simon~N.~Chandler-Wilde \hfill {\tt S.N.Chandler-Wilde@reading.ac.uk}\\
Department of Mathematics and Statistics\\
University of Reading\\
Reading, RG6 6AX\\
UK\\

\noindent
Ratchanikorn Chonchaiya \hfill {\tt ratchanikorn@buu.ac.th}\\
Department of Mathematics\\
Faculty of Science\\
Burapha University\\
Longhard Bangsaen Road\\
Muang, Chonburi\\
20131, THAILAND\\
\\
Marko Lindner\hfill {\tt lindner@tuhh.de}\\
Institute of Mathematics\\
Hamburg University of Technology\\
Schwarzenbergstr. 95 E\\
D-21073 Hamburg\\
GERMANY

\end{document}